\documentclass[final,1p,times]{amsart}

\usepackage{amssymb}
\usepackage{amsthm}
\usepackage{amscd}
\usepackage{amsmath}
\usepackage{amsfonts}
\usepackage{amssymb}
\usepackage{graphicx}
\usepackage{color}
\usepackage{framed}
\usepackage{comment}
\usepackage{chngcntr}
\numberwithin{equation}{section}
\newtheorem{theorem}{Theorem}[section]
\newtheorem{proposition}[theorem]{Proposition}
\newtheorem{lemma}[theorem]{Lemma}
\newtheorem{corollary}[theorem]{Corollary}
\newtheorem{definition}[theorem]{Definition}
\newtheorem{remark}[theorem]{Remark}
\newtheorem{example}[theorem]{Example}
\usepackage{mathrsfs}
\usepackage{titletoc}
\usepackage{tikz}
\usepackage[breaklinks,colorlinks]{hyperref}
\usepackage[pagewise,mathlines]{lineno}
\usepackage{geometry}
\usepackage{bm}
\usepackage[shortlabels]{enumitem}
\usepackage{mathtools}

\newcommand{\ra}{\rightarrow}
\newcommand{\p}{\partial}

\newcommand{\be}{\begin{equation}}
	\renewcommand{\ra}{\rightarrow}
	\newcommand{\ee}{\end{equation}}
\newcommand{\bea}{\begin{eqnarray}}
	\newcommand{\eea}{\end{eqnarray}}
\newcommand{\bna}{\begin{eqnarray*}}
	\newcommand{\ena}{\end{eqnarray*}}
\renewcommand{\o}{\omega}
\renewcommand{\O}{\Omega}
\renewcommand{\le}{\left}
\newcommand{\ri}{\right}

\newcommand{\Rmnum}[1]{\expandafter\@slowromancap\romannumeral #1@}
\makeatother

\usepackage{bbding}

\begin{document}
%
%
%
%
%
%
%
%

	\title{Nonexistence results for semilinear elliptic equations on metric graphs}
	
	\author{Yang Liu}
	\address{Yau Mathematical Sciences Center, Tsinghua University, Beijing, 100084, China}
	\email{dliuyang@tsinghua.edu.cn}
	
	\author{Yong Lin}
	\address{Yau Mathematical Sciences Center; Department of Mathematical Sciences, Tsinghua University, Beijing, 100084, China}
	\email{yonglin@tsinghua.edu.cn}
	
	\author{Haohang Zhang}
	\address{Yau Mathematical Sciences Center; Department of Mathematical Sciences, Tsinghua University, Beijing, 100084, China}
	\email{zhanghh22@mails.tsinghua.edu.cn}
	\thanks{Corresponding author: Haohang Zhang}
	
	
	
	\begin{abstract}
		In this paper, we study the nonexistence of solutions to semilinear elliptic equations with a positive potential on metric graphs. In particular, the Laplacian under consideration is of a special type, related to both the vertices and edges of metric graphs. We construct a modified distance function, introduce appropriate test functions, and establish the nonexistence of global solutions under suitable volume growth conditions imposed on the potential. More precisely, the nonnegative solutions or sign-changing solutions to the equations are the trivial zero solutions.
	\end{abstract}
	
	\maketitle
	
	\section{Introduction}
	
	Discrete or combinatorial graphs consist of sets of vertices and edges connecting these vertices. These edges primarily serve as abstract relationships between vertices or carry supplementary attributes such as weights and directions. Consequently, a function defined on a discrete graph is generally understood to be defined solely on the graph’s vertices, with no definition on its edges. Accordingly, the gradient and Laplace operator also only consider the values of the function at the vertices. In recent years, the study of partial differential equations on discrete graphs especially infinite and weighted ones has attracted significant interest. While parabolic equations have been widely explored in works like \cite{B-C,E-M,G-T,H-L,H-K-S,K-R,M2,S-S-V,L-H-N}, the elliptic setting has experienced remarkable development. Key contributions include the Schr\"{o}dinger equation \cite{C-W-Y,G-L-Y-3,Z-L-Y1,Z-Z}, the mean field equation \cite{H-L-Y,L-Yang,Liu1}, the Kazdan-Warner equation \cite{S-W,G-L-Y-2} and other equations and inequalities \cite{W1,Sun,G-L-Y-1,G-L-Y-Z,H-L-M,K-L-W,S-Y-Z-2,S-T-Z,L-W-Z}. 

In contrast, metric graphs are regarded as spatially continuous networks, where edges are treated as physical line segments joined at vertices. This continuity enables dynamic phenomena to evolve along the edges. It thus allows metric graphs to characterize the dynamical behaviors of spatial systems across numerous scientific fields, driven by distinct research motivations \cite{R,S-M-A,M-S-V,M,B-K}. Within this framework, dynamical behaviors on metric graphs are typically described by partial differential equations. These equations are defined on the edges and satisfy specific boundary conditions at the vertices. Furthermore, metric graphs can be conceived as one-dimensional manifolds with singularities. Very recently, the study of elliptic problems on metric graphs has also attracted attention from various researchers, as shown in \cite{A-S-T,B-D-S,C-J-S,D-S-T,K-M-N}.

Within traditional graph theory, the Laplacian is defined as an operator that acts on vertex-valued functions (i.e., functions defined on the vertices of the graph). However, in the context of metric graphs, attention shifts to functions defined on edges. This gives rise to the Neumann Laplacian, and the functions it acts on must satisfy certain boundary conditions:  the Kirchhoff transmission condition or the homogeneous Neumann boundary condition. In this paper, we study a new type of Laplacian defined on metric graphs, which not only  includes the difference information between vertices and their neighboring vertices but also incorporates the information of the second-order derivative on edges. Let us comment that the Laplacian $\Delta_\mathcal{G}$ we use is composed of two types of Laplacians: the vertex-based Laplacian $\Delta_\mathcal{V}$ and the edge-based Laplacian $\Delta_\mathcal{E}$. In fact, the Laplacian $\Delta_\mathcal{G}$ we adopt has long appeared in physics literature as the limiting case of a quantum wire (see \cite{K-Z,R-S} for relevant examples). Moreover, in mathematics, it has also been extended to the wave equation by Friedman-Tillich (see \cite{F-T1,F-T2,Fr}), who further developed a complete calculus framework for graphs based on both edges and vertices.

In the light of the above remarks, we investigate the elliptic equation of the form
\begin{equation}\label{s1:1}
	\Delta_\mathcal{G} u(x)+V(x) |u(x)|^{\sigma}\leq0, \quad\text{in}\quad\mathcal{G}.
\end{equation}
This equation is posed on a metric graph $\mathcal{G}$ (see Definition \ref{s0:D1}) over an infinite weighted graph $(\mathcal{V}, E, l)$, where $\mathcal{V}$ is the set of vertices, $E$ denotes the set of edges, and $l$ is the weight function. Here, the function $V: \mathcal{G} \to \mathbb{R}$ is typically referred to as potentials and is assumed to be positive, and the exponent satisfies $\sigma > 1$. Moreover, $\Delta_\mathcal{G}$ signifies the Laplacian on $\mathcal{G}$ defined by $\Delta_{\mathcal{G}}u:=	(\Delta_{\mathcal{V}}u)d\mu_{\mathcal{V}}+(\Delta_{\mathcal{E}}u)d\mu_{\mathcal{E}}$. See \eqref{s2:14} for its specific form. Rather than terming it an operator, it is more like a class of integrating factors (see Section 2.4). Herein, the equation \eqref{s1:1} is actually the counterpart of the expression
$$	\Delta_\mathcal{G} u(x)+V(x) |u(x)|^{\sigma}\le(d\mu_{\mathcal{V}}+d\mu_{\mathcal{E}}\ri)\leq0, \quad\text{in}\quad\mathcal{G},$$
in the sense of integrating factors (see Remark \ref{s2:R1} for more details).

Before outlining our results and proof methods, we first provide a brief overview of relevant results in the existing literature. The study of Problem \eqref{s1:1} has a very rich history when this problem is constructed in Euclidean space or on Riemannian manifolds rather than on graphs as demonstrated in works such as \cite{D-M,G-Sun,G-S,B-P-T}. In recent years, a large number of results have appeared concerning the nonexistence of solutions to the elliptic equation \eqref{s1:1} on combinatorial graphs. To be specific, Gu-Huang-Sun \cite{G-H-S} proposed the Assumption $(p_0)$: there exists $p_0 > 1$ such that for any $x\sim y$ in $\mathcal{V}$, 
$$\frac{\omega(x,y)}{\mu(x)}\geq\frac{1}{p_0},$$
 and illustrated that the semilinear elliptic inequality \eqref{s1:1} in the case where $V\equiv1$ and $d\mu_{\mathcal{E}}=0$ has no nontrivial nonnegative solutions in $\mathcal{V}$ when the volume growth condition 
 $$\mu(B(o,R))\lesssim R^{\frac{2\sigma}{\sigma-1}}(\ln R)^{\frac{1}{\sigma-1}}$$
holds for some $o\in \mathcal{V}$ and all sufficiently large $R$ with $\sigma>1$. Later, Monticelli-Punzo-Somaglia \cite{M-P-S} removed Assumption $(p_0)$ and introduced a more general pseudo-metric $d$ on the weighted graph. They assumed that for some $x_0\in \mathcal{V}$, $R_0>1$, $\alpha\in[0,1]$, $C>0$, there holds
\begin{equation}\label{s1:4}
	\Delta_\mathcal{V} d(x,x_0)\leq \frac{C}{d^\alpha(x,x_0)}, \quad\forall x\in \mathcal{V}\setminus B_{R_0}(x_0),
\end{equation}
and then proved that the only nonnegative solution to the inequality \eqref{s1:1} with $d\mu_{\mathcal{E}}=0$ is identically zero if the positive potential $V$ satisfies
$$\sum_{x\in B_{2R}(x_0)\setminus B_{R}(x_0)}\mu_{\mathcal{V}}(x)V^{-\frac{1}{\sigma-1}}(x)\leq CR^{\frac{(1+\alpha)\sigma}{\sigma-1}},\quad\forall\ R\geq R_0.$$ 
Recently, Meglioli-Punzo \cite{M-P2} showed that if the potential $V\geq v_0>0$ is bounded away from zero and the pseudo-metric $d$ is $q$-intrinsic, i.e., for some $q\geq1$ and $C>0$,
\begin{equation*}
	\frac{1}{\mu_{\mathcal{V}}(x)}\sum_{y\in\mathcal{V},y\sim x}\omega(x,y)d^q(x,x_0)\leq C,\quad\forall x\in\mathcal{V},
\end{equation*}
and $u$ belongs to a suitable weighted space $\ell^p_\varphi(\mathcal{V},\mu_\mathcal{V})$ (where $p\geq1$ and $\varphi$ is an exponentially decaying weight at infinity), then $u\equiv0$ is the only solution to the equation
\begin{equation}\label{s1:3}
	\Delta_{\mathcal{V}} u-Vu=0, \quad\text{in}\quad\mathcal{V}.
\end{equation} 
Subsequently, Biagi-Meglioli-Punzo \cite{B-M-P} further proved that $u\equiv0$ is the only bounded solution for the equation \eqref{s1:3} when the nonnegative potential $V$ vanishes at infinity with a certain rate and $u$ satisfies a specific volume growth condition. For other relevant works on graphs, representative studies are provided in \cite{G-W,M-Q-D,M-P,M-P-S2}.

This paper is devoted to establishing nonexistence results for nontrivial global solutions to \eqref{s1:1}, under appropriate growth conditions, with no constraint on the sign of solutions.  The main innovations of this paper are as follows:
\begin{enumerate}[(a)]
	\item To the best of our knowledge, no Liouville-type theorems for equation \eqref{s1:1} involving Laplacian $\Delta_{\mathcal{G}}$ have been explored so far. It should be noted that equation \eqref{s1:1} in \cite{G-H-S,M-P-S} is formulated in the setting of combinatorial graphs, with only the vertex Laplacian $\Delta_\mathcal{V}$ being considered. Meanwhile, the work presented in \cite{M-P} focuses on metric graphs, yet it only incorporates the edge Laplacian $\Delta_\mathcal{E}$. 
	\item Since our Laplacian is defined on both vertices and edges, substantial differences naturally arise. In particular, severe singularities appear when the distance function is defined on edges.
	\item Although our proof framework has some points in common with the counterparts in \cite{M-P-S,M-P}, many of the existing methods cannot be directly applied to our setting. Specifically, we introduce a modified distance function and construct the integration by parts formulas (see Lemmas \ref{s2:L1} and \ref{s4:L4}) within this framework. While this modified distance is not a pseudo-metric, it still satisfies \eqref{s1:4} at the vertices. 
	\item It is important to note that our results encompass the previous results when $d\mu_{\mathcal{E}}=0$ or $d\mu_{\mathcal{V}}=0$. 
\end{enumerate}

Our analysis covers two cases: nonnegative solutions (see Theorem \ref{s3:T1}) and general sign-changing solutions  (see Theorem \ref{s3:T2}). For the case of nonnegative solutions, the proof is based on a priori estimates (see Lemma \ref{s4:L1}) derived by selecting appropriate test functions. It should be stressed that all these test functions have compact support. On the other hand, this approach cannot be applied for general sign-changing solutions. Even though the proof still relies on a priori estimates (see Lemma \ref{s4:L5}) and test function selection, compactly supported test functions are insufficient in this case. To this end, we draw inspiration from \cite{M-P-S2} and use test functions supported on the entire metric graph, which have a certain exponential decay property at infinity. Finally, the upper bound estimate of the derivative of the test function in \cite{M-P-S,M-P-S2} relies on the use of a pseudo-metric.  By contrast, a further key difficulty in the present work is that our corresponding estimates depend on the derivative of the modified distance (see Lemmas \ref{s4:L2} and \ref{s4:L3}).

The remaining parts of this paper are organized as follows: In Section 2, we describe the relevant mathematical framework, focusing primarily on the concepts associated with metric graphs . In Section 3, we present the assumptions for the metric graphs considered throughout this paper, as well as the main results and their corresponding corollaries. Sections 4 is devoted to proving the results for the elliptic equation \eqref{s1:1}. These include the nonexistence of nonnegative solutions and the nonexistence of sign-changing solutions, all of which are concerned with the case of infinite metric graphs and integrating factors $\Delta_\mathcal{G}$.

	\section{Mathematical framework}

	While comprehensive definitions and results on metric graphs can be found in \cite{B-K,M,M-P}, the present section gathers core basic notions, foundational definitions, and essential preliminaries for analysis on metric graphs, all included for the reader’s convenience.

\subsection{The metric graph setting}

Like combinatorial graphs, a metric graph comprises a countable set $\mathcal{V}$ of vertices and a countable set $E$ of edges. In contrast to combinatorial graphs, however, the edges are treated as intervals glued together at the vertices. Given a function $l:E\ra(0,+\infty]$, it is usually referred to as a weight. We consider the weighted graph $(\mathcal{V}, E, l)$ and regard $l(e)$ as the length of the edge $e\in E$ (denoted as $l_e$ for short). Let 
$$\mathcal{E}:=\bigcup_{e\in E}\{e\}\times(0,l_e).$$
We may give the following definition.
\begin{definition}\label{s0:D1}
	The metric graph $\mathcal{G}$ over the weighted graph $(\mathcal{V}, E, l)$ is the pair $(\mathcal{V},\mathcal{E})$.
\end{definition}
We equip the metric graph $\mathcal{G}$ with maps $i : E \to \mathcal{V}$ assigning the initial vertex of each edge and $j : \{e \in E : l_e < +\infty\} \to \mathcal{V}$ assigning the final vertex, with these vertices collectively referred to as the endpoints of the edge. We always assume for simplicity that  $l_e<+\infty$ for all $e\in E$, and use the following notations,
$$I_e:=(0,l_e),\quad\mathcal{G}_e:=\{e\}\times I_e,\quad\overline{\mathcal{G}}_e:=e\cup\{i(e),j(e)\}\times\overline{I}_e.$$ 
For $e \in E$ and $v\in \mathcal{V}$, we write $e\ni v$ (or $v\in e$) if \(i(e)=v\) or \(j(e)=v\) (i.e., $v$ is an endpoint of $e$). In what follows, we adopt the notational convention of denoting points in $\mathcal{G}$ as $x \in \mathcal{G}$, where either $x = v \in \mathcal{V}$ or $x \in \mathcal{G}_e$ for some $e \in E$. For simplicity, we sometimes make no distinction between $e$ and $I_e$, performing this identification by abuse of language; accordingly, we may write $x\in e$ or $x\in I_e$ instead of $x\in\mathcal{G}_e$, and denote $e\in\mathcal{E}$ as $\mathcal{G}_e\in\mathcal{E}$, without causing confusion. Moreover, this practice introduces no ambiguity when the same notation $x, y, \dots$ is used to denote both points of the edge $e \in E$ and points of the interval $I_e \subseteq \mathbb{R}_+$. For each $e \in \mathcal{E}$, the map $\pi_e \colon \mathcal{G}_e \to I_e$ defined by $\pi_e(\{e\}, x) \equiv \pi_e(x) := x$ sets up a bijection between points of $e\in E$ and points of $I_e$. This map can be extended to a mapping from $\overline{\mathcal{G}}_e$ to $\overline{I}_e = [0, l_e]$ such that \(\pi_e(i(e)) = 0\) and \(\pi_e(j(e)) = l_e\).

\begin{definition}
	Let $\mathcal{G}$ be a metric graph.
	
	(i) A metric graph $\mathcal{G}$ is finite if both $E$ and $\mathcal{V}$ are finite sets; it is infinite otherwise.
	
	(ii) For a vertex $v\in\mathcal{V}$,  its degree $\deg_v \in \mathbb{N}$ counts the number of edges $e \ni v$. The inbound degree $\deg_v^+$ (resp. outbound degree $\deg_v^-$) refers to the number of edges with $j(e) = v$ (resp. $i(e) = v$). Obviously, $\deg_v=\deg_v^++\deg_v^-$. $\mathcal{G}$ is locally finite if $\deg_v < \infty$ for all $v \in \mathcal{V}$. 
	
	(iii) For two vertices $u, v \in \mathcal{V}$, a path connecting them is a set $\{x_1, \dots, x_n\} \subset \mathcal{G}$ ($n \in \mathbb{N}$) such that $x_1 = u$, $x_n = v$, and for each $k = 1, \dots, n-1$, there exists an edge $e_k$ where both $x_k$ and $ x_{k+1}$ lie in $ \overline{\mathcal{G}}_{e_k}$. A path is closed if its start and end vertices coincide $(u = v)$. A closed path is termed a cycle if it does not pass through the same vertex more than once.
	
	(iv) A metric graph $\mathcal{G}$ is connected if there exists a path between any two distinct vertices $v, w \in \mathcal{V}$. A connected graph with no cycles is called a tree.
	
	(v) The boundary of the metric graph is given by $\p\mathcal{G} := \{v\in\mathcal{V}\mid \deg_v=1\}$.
\end{definition}

\subsection{Two volume measures}

In traditional combinatorial analysis, concepts such as integrals, Laplacians, and Rayleigh quotients are all defined using a single volume measure. In this paper, we depart from this convention by employing two distinct volume measures.

We first define an edge measure. A connected metric graph $\mathcal{G}$ can naturally be endowed with the structure of a metric measure space. To elaborate, for any two points $x, y \in \mathcal{G}$, we may treat them as vertices of a connecting path $P$ (with $x$ and $y$ possibly added to the vertex set $\mathcal{V}$ if necessary). The length of $P$ is defined as the sum of the length of its $n$ edges $e_k$, i.e., $l(P) := \sum_{k=1}^n l_{e_k}$. The distance $ d(x, y)$ between $x$ and $y$ is then given by the infimum of the lengths of all such connecting paths:
$$d(x, y) \coloneqq \inf \left\{ l(P) \mid P \text{ connects } x \text{ and } y \right\}.$$
This makes $\mathcal{G}$ a metric space, which in turn induces a topological structure via the metric topology. Let $\mathcal{B} = \mathcal{B}(\mathcal{G})$ denote the Borel $\sigma$-algebra of $\mathcal{G}$. Let $B(x_0,r)$ denote the open ball on the metric graph $\mathcal{G}$ with center $x_0\in\mathcal{G}$ and radius $r>0$, which consists of all points in $\mathcal{G}$ whose distance from $x_0$ is less than $r$. If $\mathcal{G}$ is locally finite, then $B(x_0,r)$ is a union of finitely many open subintervals of edges and finitely many entire edges for $r$ small enough. Hence, a Radon measure $\mu_{\text{Rad}} \colon \mathcal{B} \to [0, \infty]$ on $\mathcal{G}$ is induced via the Lebesgue measure $\lambda$ on each interval $I_e$, specifically,
\begin{equation}\label{s2:8}
	\mu_{\text{Rad}}(\O):=\sum_{e\in\mathcal{E}}\lambda(I_e\cap \O),\quad\forall\ \O\in\mathcal{B}.
\end{equation}
\begin{definition}\label{s2:D3}
	Let $\mu_{\mathcal{V}}:\mathcal{V}\ra\mathbb{R}$ be a vertex measure supported on the vertex set $\mathcal{V}$, with $\mu_{\mathcal{V}}(v) > 0$ for every $v \in \mathcal{V}$. Moreover, the Radon measure $\mu_{\text{Rad}}$ naturally defines an edge measure $\mu_{\mathcal{E}}:\mathcal{E}\ra(0,l_e]$, satisfying $\mu_{\mathcal{E}}(v)=0$ for all $v\in\mathcal{V}$. 
\end{definition}

We denote by $\mathcal{F}$ the set of all functions $f : \mathcal{G} \to \mathbb{R}$. For any $f \in \mathcal{F}$, we let $f_e := f|_{\overline{I}_e}$. Every function $f \in \mathcal{F}$ thus canonically induces a countable family of functions $\{f_e\}_{e\in \mathcal{E}}$, where \(f_e: \overline{I}_e \to \mathbb{R}\), and we accordingly write
\begin{equation*}
	f=\bigoplus_{e\in\mathcal{E}}f_e.
\end{equation*}
We define $f^{(h)} \coloneqq \bigoplus_{e\in \mathcal{E}} f^{(h)}_e$ for $h \in \mathbb{N}$, where the derivative $f^{(h)}_e \coloneqq \frac{d^h f_e}{dx^h}$ exists on $I_e$ for all $e \in E$. We also adopt the notation $f^{(0)} \equiv f$, $f^{(1)} \equiv f^\prime$ and $f^{(2)} \equiv f^{\prime\prime}$. We say that $f$ is continuous on $\mathcal{G}$, writing $f \in C(\mathcal{G})$, if $f_e \in C(\overline{I}_e)$ for all $e \in E$ and, at every vertex $x$, $f_e(x)$ coincides for all $e\ni x$. We set
$$C^k(\mathcal{G}):=\{f\in C(\mathcal{G})\mid f_e\in C^{(k)}(\overline{I}_e), \forall e\in\mathcal{E}, f^{(h)}\in C(\mathcal{G}),\forall h=1,\dots,k\},\quad(k\in\mathbb{N}),$$
and $C^0(\mathcal{G})=C(\mathcal{G})$. We further denote $\mathcal{F}_0$ as the subspace of  $\mathcal{F}$ consisting of functions where $f_e(x)$ is consistent for every $x\in\mathcal{V}$ and all $e\ni x$.

Based on \eqref{s2:8}, for any measurable function $f\in \mathcal{F}_0$, we set
\begin{equation}\label{s2:11}
	\int_\mathcal{G}f(x) d\mu_{\mathcal{E}}:=\sum_{e\in \mathcal{E}}\int_0^{l_e} f_e(x)dx,\quad \int_{\Omega} f(x)d\mu_{\mathcal{E}}:=\int_\mathcal{G}f(x) \textbf{1}_{\Omega}(x)d\mu_{\mathcal{E}},\quad\forall\ \Omega\in\mathcal{B}.
\end{equation}
Here, $\textbf{1}_{\O}$ denotes the characteristic function of the set $\O$, and we use the standard notation $dx \equiv d\lambda$. For integrals over the vertex set $\mathcal{V}$, we still adhere to the integral form used in combinatorial graphs, i.e.
\begin{equation*}
    \int_\mathcal{G}f(x) d\mu_{\mathcal{V}}:=\sum_{x\in\mathcal{V}}\mu_{\mathcal{V}}(x)f(x),\quad \int_{\Omega} f(x)d\mu_{\mathcal{V}}:=\int_\mathcal{G}f(x) \textbf{1}_{\Omega}(x)d\mu_{\mathcal{V}},\quad\forall\ \Omega\in\mathcal{B}.
\end{equation*}
The vertex-based and edge-based integral forms as above determine an integral over the metric graph:
\begin{equation}\label{s2:22}
	\int_{\mathcal{G}}f(x)d\mu_{\mathcal{G}}:=	\int_{\mathcal{G}}f(x)\left(d\mu_{\mathcal{V}}+d\mu_{\mathcal{E}}\right)=\int_\mathcal{G}f(x) d\mu_{\mathcal{V}}+	\int_\mathcal{G}f(x) d\mu_{\mathcal{E}},\quad\forall\ f\in \mathcal{F}_0.
\end{equation}

For $1\leq p\leq\infty$, we define $\ell^p(\mathcal{V}):=\left\{f\in \mathcal{F}_0:\|f\|_{\ell^p(\mathcal{V})}<\infty\right\}$, with the norm
	\begin{equation*}
	\|f\|_{\ell^p(\mathcal{V})}= \left\{\begin{aligned}
		&\left(\sum_{x\in \mathcal{V}}\mu_{\mathcal{V}}(x)|f(x)|^p\right)^{\frac{1}{p}}, &1\leq p<\infty,\\
		&	\sup_{x\in\mathcal{V}}|f(x)|, &p=\infty.\ \ \ \ \ \ \,\end{aligned}\right.
\end{equation*}
For each $p \in [1, \infty]$, the Lebesgue spaces on the metric graph $\mathcal{G}$ (denoted $L^p(\mathcal{G})$ or $L^p(\mathcal{G}, \mu_\mathcal{G}$) take the form
$$L^p(\mathcal{G}):=\le(\bigoplus_{e\in\mathcal{E}}L^p(I_e,\lambda)\ri)\oplus \ell^p(\mathcal{V}),$$
endowed with the corresponding norm
$$\|f\|_p:=\|f\|_{p,\mathcal{V}}+\|f\|_{p,\mathcal{E}}=\|f\|_{\ell^p(\mathcal{V})}+\sum_{e\in\mathcal{E}}\|f_e\|_p=\left(\sum_{x\in \mathcal{V}}\mu_{\mathcal{V}}(x)|f(x)|^p\right)^{\frac{1}{p}}+\sum_{e\in\mathcal{E}}\le(\int_0^{l_e}|f_e|^pdx\ri)^\frac{1}{p},\quad\forall\ p\in[1,\infty),$$
$$\|f\|_{\infty}:=\text{ess} \sup_{x\in\mathcal{G}}|f(x)|.$$

We further let $\varphi: \mathcal{G} \to \mathbb{R}$ be a positive continuous function. For each \(p \in [1, +\infty)\), we define the weighted Lebesgue space $L_\varphi^p(\mathcal{G})$ as follows:
\begin{equation}\label{s2:13}
	L_\varphi^p(\mathcal{G}):=\left(\bigoplus_{e \in \mathcal{E}} L_\varphi^p(I_e, \lambda) \right)\oplus \ell^p_{\varphi}(\mathcal{V})= \left\{ f \in \mathcal{F}_0 \mid \left(\sum_{x\in \mathcal{V}}\mu_{\mathcal{V}}(x)\varphi(x)|f(x)|^p\right)^{\frac{1}{p}}+\sum_{e \in \mathcal{E}} \left( \int_0^{l_e} |f_e|^p \varphi_e \, dx \right)^{\frac{1}{p}} < +\infty \right\}.
\end{equation}

\subsection{The Laplacian on metric graphs}

Let $\o:E\rightarrow \mathbb{R}^+$ be a positive symmetric weight satisfying ${\o(x,y)>0}$ and $\o(x,y)=\o(y,x)$ for every edge ${(x,y)\in E}$. For any $f\in \mathcal{F}_0$, the vertex-based Laplacian on $\mathcal{G}$ is defined as
\begin{equation}\label{s2:7}
	\Delta_\mathcal{V} f(x)=\frac{1}{\mu_\mathcal{V}(x)}\sum_{y\sim x}\o(x,y)(f(y)-f(x)),
\end{equation}
where $y\sim x$ means $y$ is adjacent to $x$, i.e., $(x,y)\in E$. This is the usual combinatorial Laplacian. It is not difficult to see that the following integration by parts formula is valid:
\begin{equation}\label{s2:18}
	\int_\mathcal{G}f(\Delta_{\mathcal{V}}g)d\mu_{\mathcal{V}}=	\int_\mathcal{G}g(\Delta_{\mathcal{V}}f)d\mu_{\mathcal{V}},
\end{equation}
provided that at least one of the functions $f,g\in \mathcal{F}_0$ has finite support.

We consider a space
\begin{equation}\label{s2:2}
	\mathcal{D}(\mathcal{G}):=\left\{f\in C(\mathcal{G}): f_e\in C^2(I_e)\cap C^1(\overline{I}_e), f_e^{\prime\prime}\in L^\infty(I_e), \forall\ e\in\mathcal{E}\right\},
\end{equation}
and note that if $f \in \mathcal{D}(\mathcal{G})$, then $f$ is in $C(\mathcal{G})$, but generally not in $C^1(\mathcal{G})$; specifically, for a vertex $v \in \overline{\mathcal{G}}_{e_1} \cap \overline{\mathcal{G}}_{e_2}$ (where \(e_1, e_2 \in E\)), it may hold that $f^\prime_{e_1}(v) \neq f^\prime_{e_2}(v)$. 

Using the functional framework introduced above, the metric graph $\mathcal{G}$ can be endowed with a  edge-based Laplacian $\Delta_\mathcal{E}$, an operator that acts on $\mathcal{D}(\mathcal{G})$ in the canonical way
\begin{equation}\label{s2:12}
	\Delta_{\mathcal{E}}  f(x):=f^{\prime\prime}_e(x),\quad\forall\ f\in \mathcal{D}(\mathcal{G}),\ e\in\mathcal{E},\ x\in I_e.
\end{equation}
The outer normal derivative of $f_e$ at a vertex $v \in \mathcal{V}$ is denoted by
\begin{equation}\label{s2:3}\frac{df_e(v)}{dn}=\left\{\begin{array}{lll}
		f_e^\prime(v), &\text{if}\ j(e)=v,&\\[1ex]
		-f_e^\prime(v),&\text{if}\ i(e)=v.&\end{array}\ri.		
\end{equation}
For any $x\in\mathcal{V}$, we define 
\begin{equation}\label{s2:4}
	[\mathcal{K}(f)](x):=\sum_{e\ni x}\frac{df_e(x)}{dn}.
\end{equation}

 Given that we have introduced two  Laplacians corresponding to measures $\mu_{\mathcal{V}}$ and $\mu_{\mathcal{E}}$, it is thus necessary to define the Laplacian $\Delta_{\mathcal{G}}$ on the metric graph from the perspective of integrating factors. Specifically, we need to mark functions with $d\mu_{\mathcal{V}}$ or $d\mu_{\mathcal{E}}$ to clarify how the function should be integrated against other functions. In this paper, we use a similar setting as in \cite{F-T2,Fr} and always consider the Laplacian $\Delta_\mathcal{G}$ of the form
 \begin{equation}\label{s2:14}
 	\Delta_{\mathcal{G}}f:=	(\Delta_{\mathcal{V}}f)d\mu_{\mathcal{V}}+(\Delta_{\mathcal{E}}f)d\mu_{\mathcal{E}},\quad\forall f \in \mathcal{D}(\mathcal{G}),
 \end{equation}
 where $\Delta_{\mathcal{V}}$ and $\Delta_{\mathcal{E}}$ are defined in accordance with \eqref{s2:7} and \eqref{s2:12}. According to \cite{F-T1}, as an integrating factor, $\Delta_{\mathcal{G}}f$ can generate a linear functional $\mathcal{L}_{\Delta_{\mathcal{G}}f}$ on $\mathcal{F}_0$ by means of 
\begin{equation}\label{s2:15}
\mathcal{L}_{\Delta_{\mathcal{G}}f}(g):=\int_\mathcal{G} g \Delta_{\mathcal{G}}f=\int_\mathcal{G}g(\Delta_{\mathcal{V}}f)d\mu_{\mathcal{V}}+\int_\mathcal{G}g(\Delta_{\mathcal{E}}f)d\mu_{\mathcal{E}}.
\end{equation}

\subsection{Definition of solutions}

We now give the definition of a solution to the equation \eqref{s1:1}.
\begin{definition}
	We say that $u\in\mathcal{D}(\mathcal{G})$ is a solution of \eqref{s1:1} whenever, for each $e\in\mathcal{E}$,
	\begin{equation*}
		u_e^{\prime\prime}(x)+V_e(x)|u_e(x)|^\sigma\leq0,\quad\forall\ x\in(0, l_e),
	\end{equation*} 
	and, for each $x\in\mathcal{V}$,
	\begin{equation*}
	   \Delta_\mathcal{V} u(x) +V(x) |u(x)|^\sigma\leq0,
	\end{equation*} 
	with 
	\begin{equation}\label{s2:19}
		[\mathcal{K}(u)](x)= 0.
	\end{equation}
	Furthermore, $u$ is called a nonnegative solution if $u(x) \geq 0$ for all $x \in \mathcal{G}$.
\end{definition}

For interior vertices $x\in \mathcal{G}\setminus\p\mathcal{G}$, the condition $ [\mathcal{K}(u)](x)= 0$ is referred to as the Kirchhoff transmission condition; for boundary vertices $x \in \partial\mathcal{G}$, this condition corresponds to the homogeneous Neumann boundary condition. 

It can be seen that, in comparison with solutions of equation \eqref{s1:3} on combinatorial graphs, solutions to \eqref{s1:1} on metric graphs not only satisfy the vertex-wise solution properties of combinatorial graphs but are additionally required to be defined on edges  and must satisfy certain boundary conditions at vertices. 

\begin{remark}\label{s2:R1}
	If $u\in \mathcal{D}(\mathcal{G})$ is a solution of \eqref{s1:1}, then it must satisfy 
	$$\Delta_\mathcal{V}u(x)+V(x)|u(x)|^\sigma\leq0,\quad\Delta_\mathcal{E}u(x)+V(x)|u(x)|^\sigma\leq0,$$
	and \eqref{s2:19}. Hence, in the sense of integrating factors, it fulfills 
		\begin{equation}\label{s2:21}
		\Delta_\mathcal{G} u(x)+V(x) |u(x)|^\sigma\le(d\mu_{\mathcal{V}}+d\mu_{\mathcal{E}}\ri)\leq0.
	\end{equation}
	On the other hand, most existing results only considered one of the two inequalities mentioned above. The first inequality has been investigated for combinatorial graphs in \cite{G-H-S,M-P-S}, whereas a variant of the second inequality has been analyzed for metric graphs in \cite{M-P}.
\end{remark}

\section{Statement of the main results}

		In this section, we are going to state our main theorems concerning the Laplacian $\Delta_{\mathcal{G}}$ on a metric graph $\mathcal{G}$. For any $x_0\in\mathcal{V}$ and $R>0$, we denote by
	$$B_R(x_0)=\le\{x\in\mathcal{G}\mid d(x,x_0)<R\ri\}$$
	the ball of radius $R > 0$ centered at $x_0$.  In the sequel, we always make the following hypothesis:
	\begin{equation}\label{s3:1}
		\begin{cases}
			\text{ (i) }	\mathcal{G} \text{ is an infinite, connected, locally finite metric graph.}\\
		\text{ (ii) }	\text{For all } e\in\mathcal{E}, l_e<+\infty \text{  and } i(e)\not=j(e)  \text{ (there are no loops or rays). }\\
		\text{ (iii) }	\text{There exists a vertex } x_0\in\mathcal{V} \text{ and a sequence } \{x_n\}\subset \mathcal{V}  \text{ such that } d(x_n, x_0) \text{ tends to } +\infty. \\
		\text{ (iv) }	\text{ For any } R>0, \text{ the set } B_R:=B_R(x_0) \text{ is finite, consisting of finitely many } v\in\mathcal{V}\text{ and } e\in \mathcal{E}.\\
		\text{ (v) }	\text{ There exists a constant } C > 0	\text{  such that for every } x\in\mathcal{V}, \sum_{y\sim x}\o(x,y)\leq C\mu_{\mathcal{V}}(x).\\
		\text{ (vi) }	\text{ Suppose }j:=\sup_{e\in \mathcal{E}}l_e<\infty,\text{ and } r:=\inf_{e\in\mathcal{E}}l_e>0.
		\end{cases}
	\end{equation}
	
	Fix $x_0\in\mathcal{V}$. For any $x\in \mathcal{G}$, we recall the definition of the distance $d(x,x_0)$. If $w\in \mathcal{V}$, then there exists a shortest path $P:x_0\sim x_1\sim x_2 \sim \cdots\sim x_k=w$, such that 
	$$d(w,x_0)=\sum_{n=1}^kl_{e_n},$$ 
	where $e_n$ denotes the edge between $x_{n-1}$ and $x_n$ for all $n=1,\cdots, k$. In this case, for every edge $e\ni w$, the one-sided derivative of $d_e(\cdot,x_0)$ along $e$ at $w$ is $1$ or $-1$. On the other hand, if $x\in e$, we let $x$ lie in the interior of edge $e$ with endpoints $x_l$ and $x_r$. Then the distance is given by
	$$d(x,x_0)=\left\{\begin{array}{lll}
		\min\{d(x_l,x_0)+x,d(x_r,x_0)+l_e-x\}, &\text{if}\ \pi_e(x_l)=0\ \text{and}\ \pi_e(x_r)=l_e,&\\[1ex]
		\min\{d(x_l,x_0)+l_e-x,d(x_r,x_0)+x\}, &\text{if}\ \pi_e(x_l)=l_e\ \text{and}\ \pi_e(x_r)=0.&\end{array}\ri.	$$ 
	Owing to this minimality property of the distance function on the interior of edges,
	the distance from an interior point to the fixed vertex $x_0$ exhibits a more complex structure. Specifically, as a point $x$ moves along edge $e$ from $i(e)$ to $j(e)$, one of the following three cases arises:
	\begin{itemize}
		\item $d(x,x_0) = d\big(i(e),x_0\big) + x$;
		\item $d(x,x_0) = d\big(j(e),x_0\big) + l_e - x$;
		\item there exists a point $q_e\in e$ such that
		\[
		d\big(i(e),x_0\big) + q_e
		= d\big(j(e),x_0\big) + l_e - q_e,
		\]
		and
		\[
		d(x,x_0)=
		\begin{cases}
			d\big(i(e),x_0\big) + x, & x\in[i(e),q_e],\\
			d\big(j(e),x_0\big) + l_e - x, & x\in[q_e,j(e)].
		\end{cases}
		\]
	\end{itemize}
	Here and in the sequel, we identify points on edge $e$ with points in interval $I_e$ and use them interchangeably. In the last case, the derivative of the distance function fails to exist at some interior points of edges. This occurs because the left-hand derivative at such points is $1$ while the right-hand derivative is $-1$. It is straightforward to verify that each edge contains at most one such interior point. We regard these interior points of edges (where the distance function is non-differentiable) as additional vertices and collect them together to form a vertex set $\mathcal{V}_0$. Clearly, $\mathcal{V}\cap\mathcal{V}_0=\emptyset$. 
	
	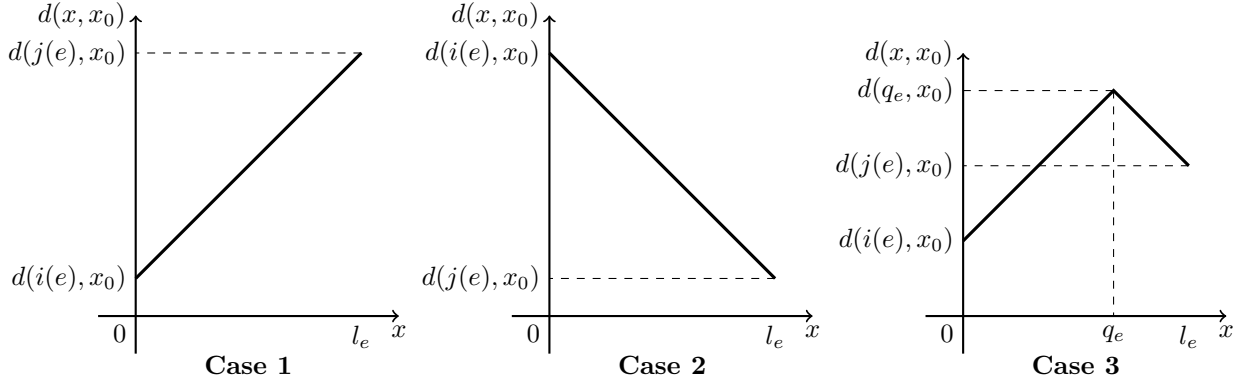
\begin{figure}[htbp]
		\centering
		\begin{tikzpicture}
			
			\begin{scope}[xshift=0cm, local bounding box=caseOne]
				\def\lE{3.0}     
				\def\dU{0.5}     
				\def\dV{3.5}     
				
				\draw[->, thick] (-0.5,0) -- (\lE + 0.5,0) node[below] {$x$};
				\draw[->, thick] (0,-0.5) -- (0,\dV + 0.5) node[left] {$d(x, x_0)$};
				
				\coordinate (startPoint) at (0, \dU);
				\coordinate (endPoint) at (\lE, \dV);
				
				\draw[very thick] (startPoint) -- (endPoint);
				
				\node[below left] at (0,0) {$0$};
				\node[below] at (\lE,0) {$l_e$};
				
				\draw[dashed] (startPoint) -- (0,\dU) node[left] {$d(i(e),x_0)$};
				\draw[dashed] (endPoint) -- (0,\dV) node[left] {$d(j(e),x_0)$};
				
				
				\node[below, font=\bfseries] at (\lE/2.0, -0.4) {Case 1};
			\end{scope}

			\begin{scope}[xshift=5.5cm, local bounding box=caseTwo]
				\def\lE{3.0}     
				\def\dU{3.5}     
				\def\dV{0.5}     
				
				\draw[->, thick] (-0.5,0) -- (\lE + 0.5,0) node[below] {$x$};
				\draw[->, thick] (0,-0.5) -- (0,\dU + 0.5) node[left] {$d(x, x_0)$};
				
				\coordinate (startPoint) at (0, \dU);
				\coordinate (endPoint) at (\lE, \dV);
				
				\draw[very thick] (startPoint) -- (endPoint);
				
				\node[below left] at (0,0) {$0$};
				\node[below] at (\lE,0) {$l_e$};
				
				\draw[dashed] (startPoint) -- (0,\dU) node[left] {$d(i(e),x_0)$};
				\draw[dashed] (endPoint) -- (0,\dV) node[left] {$d(j(e),x_0)$};
				
				
				\node[below, font=\bfseries] at (\lE/2.0, -0.4) {Case 2};
			\end{scope}

			\begin{scope}[xshift=11cm, local bounding box=caseThree]
				\def\lE{3.0}     
				\def\dU{1}     
				\def\dV{2}     
				
				\def\qE{2} 
				\def\maxD{3}
				
				\draw[->, thick] (-0.5,0) -- (\lE + 0.5,0) node[below] {$x$};
				\draw[->, thick] (0,-0.5) -- (0,\maxD + 0.5) node[left] {$d(x, x_0)$};
				
				\coordinate (startPoint) at (0, \dU);
				\coordinate (cutPoint) at (\qE, \maxD);
				\coordinate (endPoint) at (\lE, \dV);
				
				\draw[very thick] (startPoint) -- (cutPoint) -- (endPoint);
				
				\node[below left] at (0,0) {$0$};
				\node[below] at (\lE,0) {$l_e$};
				
				\draw[dashed] (startPoint) -- (0,\dU) node[left] {$d(i(e),x_0)$};
				\draw[dashed] (endPoint) -- (0,\dV) node[left] {$d(j(e),x_0)$};
				
				\draw[dashed] (cutPoint) -- (\qE, 0) node[below] {$q_e$};
				\draw[dashed] (cutPoint) -- (0, \maxD) node[left] {$d(q_e,x_0)$};
				
				
				\node[below, font=\bfseries] at (\lE/2.0, -0.4) {Case 3};
			\end{scope}
			
		\end{tikzpicture}
		\caption{Three cases of the distance $d(x, x_0)$ for $x$ moving along the edge. The last case causes singularity of derivative.}
		\label{fig:distanceFunctions}
	\end{figure}
	
	Since boundary terms arise in the integration by parts formula and the distance function $d$ is non-differentiable at singular points $q_e\in\mathcal{V}_0$, we introduce a modified distance function $\tilde{d}$ via mollification (detailed in Subsection 4.1). This ensures that the derivative of the modified distance function vanishes at the midpoints of all edges containing singular points, while its one-sided derivatives also vanish at every $x\in\mathcal{V}$. We now present the nonexistence results concerning the elliptic equation \eqref{s1:1}. Specifically, our goal is to show that, under suitable assumptions, the global nonnegative solution of \eqref{s1:1} is the identically zero solution. 
		
		\begin{theorem}\label{s3:T1}
		Let $\mathcal{G}$ be a metric graph satisfying \eqref{s3:1}. Suppose that $u$ is a nonnegative solution of \eqref{s1:1} with $\sigma>1$, and that the potential $V:\mathcal{G}\ra\mathbb{R}$ is a positive function satisfying 
		\begin{equation}\label{s3:2}
			\sum_{x\in E_R\cap \mathcal{V}} \mu_{\mathcal{V}}(x)V^{-\frac{1}{\sigma-1}}(x)\leq CR^{\frac{\sigma}{\sigma-1}},\quad\sum_{e \in E_R\cap\mathcal{E}}\int_{0}^{l_e} V_e^{-\frac{1}{\sigma-1}}(x)dx\leq CR^{\frac{\sigma}{\sigma-1}},
		\end{equation}
		for some constant $C>0$ and every $R\geq R_0=\max\{2j,1\}$, where \begin{equation}\label{s3:3}
			E_R=\{x\in\mathcal{G}:R\leq d(x,x_0)\leq 2R\}.
		\end{equation}
		Then $u\equiv0$ on $\mathcal{G}$.
	\end{theorem}
	
		As an immediate consequence of Theorem \ref{s3:T1}, setting $V \equiv 1$ on $\mathcal{G}$ yields the following corollary.
	\begin{corollary}
		Let $\mathcal{G}$ be a metric graph satisfying \eqref{s3:1} and $\sigma>1$. Assume that there exists a constant $C>0$ such that for all $R\geq R_0=\max\{2j,1\}$,
		\begin{equation*}
			\mu_\mathcal{V}(E_R)\leq CR^{\frac{\sigma}{\sigma-1}},\quad\mu_\mathcal{E}(E_R)\leq CR^{\frac{\sigma}{\sigma-1}},
		\end{equation*}
		where $\mu_\mathcal{V}$ and $\mu_\mathcal{E}$ are two measures given by Definition \ref{s2:D3}, and $E_R$ is defined in \eqref{s3:3}. If $u$ is a nonnegative solution of 
		$$\Delta_\mathcal{G} u+u^{\sigma}\leq0, \quad\text{in}\quad\mathcal{G},$$
		then $u\equiv0$  on $\mathcal{G}$.
	\end{corollary}
	
For given constant $\delta>0$ and fixed $R\geq R_0$, we always write $\alpha=\delta/R$ and
\begin{equation}\label{s3:16}
	X_\alpha:=L^1_{e^{-\alpha d(x,x_0)}}(\mathcal{G}),
\end{equation}
where $L^1_{e^{-\alpha d(x,x_0)}}(\mathcal{G})$ is a weighted Lebesgue space defined in \eqref{s2:13}. The proof of Theorem \ref{s3:T1} relies on a priori estimates obtained via appropriate compactly supported test functions. Nevertheless, when extending our analysis to sign-changing solutions, the test functions we employ are supported on the entire graph and exhibit sufficiently rapid decay at infinity. This forces us to impose additional constraints, such as a more stringent weighted volume growth condition on  the potential $V$, as well as the requirement that the solution $u$ lies in a suitable weighted space $X_\alpha$. Collectively, these form the exact content of our next theorem.

\begin{theorem}\label{s3:T2}
	Let $\mathcal{G}$ be a metric graph satisfying \eqref{s3:1}, and let $\sigma>1$. Suppose that the potential $V:\mathcal{G}\ra\mathbb{R}$ is a positive function satisfying 	
			\begin{equation}\label{s3:4}
			\sum_{x\in\mathcal{V}\cap B_{R}^c}\mu_{\mathcal{V}}(x) V^{-\frac{1}{\sigma-1}}(x)e^{-\alpha d(x,x_0)}\leq CR^{\frac{\sigma}{\sigma-1}},\quad \sum_{e\in\mathcal{E}\cap B_{R}^c}\int_{0}^{l_e} V_e^{-\frac{1}{\sigma-1}}(x) e^{-\alpha d(x,x_0)}dx\leq CR^{\frac{\sigma}{\sigma-1}}, \quad\forall\ R\geq R_0,
		\end{equation}
		for some $\alpha=\delta/R>0$ and $C>0$. If $u$ is a solution of \eqref{s1:1} and $u\in X_\alpha$, then $u\equiv0$ on $\mathcal{G}$.
\end{theorem}

We emphasize that the distance used in Theorems \ref{s3:T1} and \ref{s3:T2} is the original distance function $d$, not the modified one. In other words, the modified distance $\tilde{d}$ is only utilized in the proofs of the theorems. This is a desirable feature for results on metric graphs, as our conclusions only depend on the natural structure of the metric graph, not on the choice of the modified distance function.

		\section{Proofs of the main results for the elliptic equation}
		
			In this section, we first provide a modified distance function and derive a priori estimates for solutions to the equation \eqref{s1:1} by constructing two test functions of different forms. We finally prove Theorems \ref{s3:T1} and \ref{s3:T2} in sequence.
		
		\subsection{Modified distance functions}
		Let $\eta:[0,1]\to[0,1]$ be a $C^3$-function satisfying
		\begin{equation}\label{s4:4}
			\begin{cases}
				\eta(0)=0,\ \eta(1)=1;\\
				\eta^\prime_+(0)=\eta^\prime_-(1)=0,\ \eta^{\prime\prime}_+(0)=\eta^{\prime\prime}_-(1)=0;\\
					\eta^\prime(x)\geq0, \text{ for all}\ x\in (0, 1);\\
				\text{there exists}\ C > 0\ \text{such that}\ |\eta^\prime(x)|\leq C\ \text{and}\  |\eta^{\prime\prime}(x)|\leq C, \text{ for all}\ x\in (0, 1).
			\end{cases}
		\end{equation}
		We introduce  a modified distance function $\tilde{d}(x,x_0):\mathcal{G}\to\mathbb{R}$ by
		\begin{equation}\label{s4:36}
			\tilde{d}=\bigoplus_{e\in\mathcal{E}}\tilde{d}_e,\quad\tilde{d}_e(x,x_0):=d(\tilde{x}_e(x),x_0),
		\end{equation}
		where $\tilde{x}_e:\overline{I}_e\to\overline{I}_e$ denotes the coordinate transformation. For each edge $e\in\mathcal{E}$, we define $\tilde{x}_e$ as follows:
		\begin{enumerate}[(i)]
			\item if $e \cap \mathcal{V}_0 = \emptyset$, 
			\begin{equation}\label{s4:33}
				\tilde{x}_e(x)=l_e\eta\le(\frac{x}{l_e}\ri), \quad x\in\le[0,l_e\ri];
			\end{equation}
			\item if $e \cap \mathcal{V}_0=\{q_e\}$, 
			\begin{equation}\label{s4:34}
				\tilde{x}_e(x)=
				\begin{cases}
					q_e\eta\left(\frac{2x}{l_e}\right), & x\in\le[0,\frac{l_e}{2}\ri],\\[1.5ex]
					(l_e-q_e)\eta\left(\frac{2x-l_e}{l_e}\right)+q_e, & x\in\le[\frac{l_e}{2},l_e\ri].\\[1.5ex]
				\end{cases}
			\end{equation}
		\end{enumerate}
		Hence, $\tilde{d}(w,x_0)=d(w,x_0)$ for all $w\in \mathcal{V}$, and $\tilde{d}_e(l_e/2,x_0)=d_e(q_e,x_0)$ for each $q_e\in \mathcal{V}_0$. This is a regularization for edges with sigularity, which are repositioned to the middle point. Moreover, it is obvious that
		\begin{equation}\label{s4:22}
			|d(x,x_0)-\tilde{d}(x,x_0)|\leq j,\quad\forall\ x\in\mathcal{G}.
		\end{equation}

		The advantage of the modified coordinate transformation $\tilde{x}_e$ is that it smooths the
	distance function at the singular point $q_e$ where the original distance function is not differentiable. Consequently, the modified distance function $\tilde{d}(x,x_0)$ becomes twice continuously differentiable across the entire edge $e$ without singularities. Importantly, the modified distance function coincides with the distance function at all original vertices in $\mathcal{V}$, but its one-sided derivatives at these vertices are zero, rather than the original $\pm1$. Hence, the modified distance function $\tilde{d}\in C^2(\mathcal{G})$. These are precisely the properties described below.

			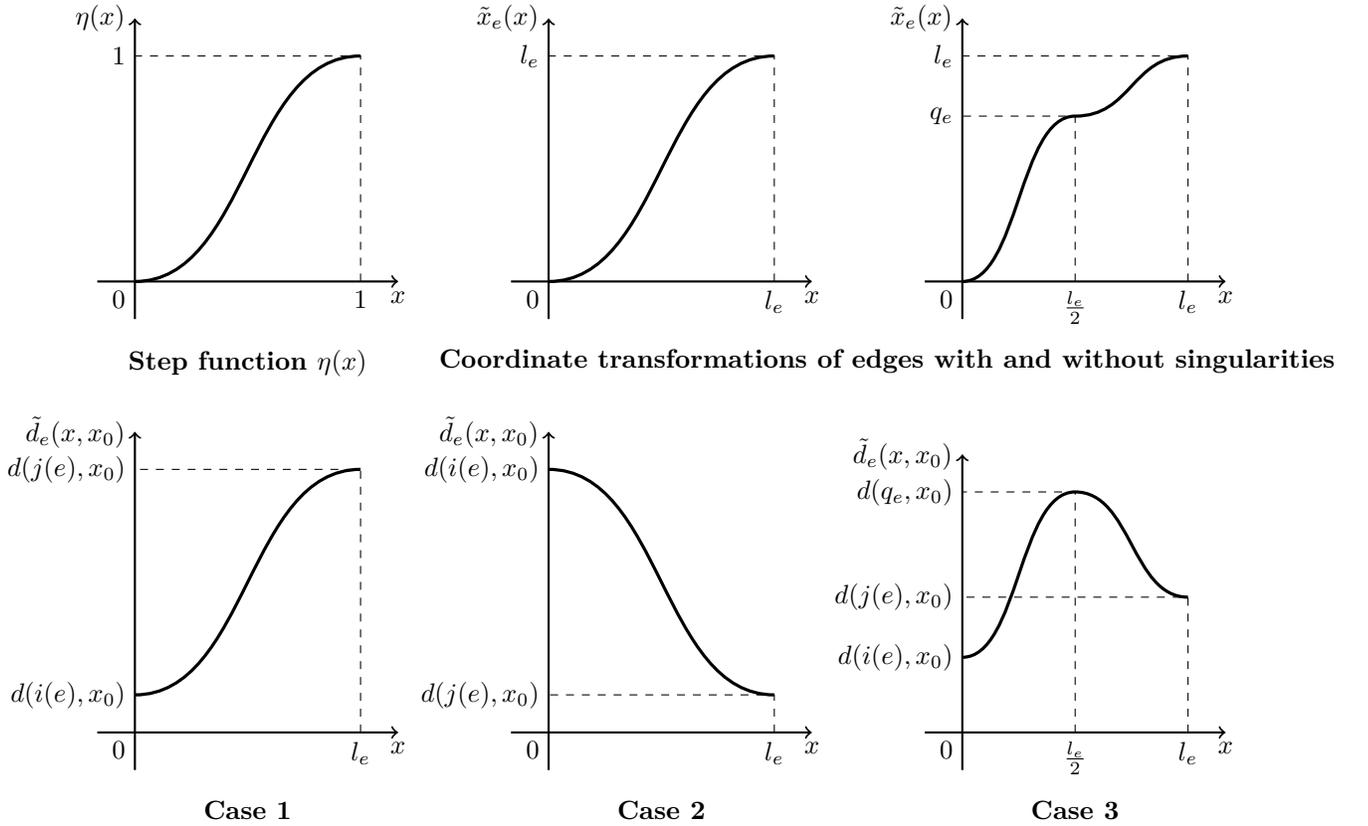
\begin{figure}[htbp]
				\centering
				\begin{tikzpicture}
					
					\def\rowOneY{0cm}
					
					\begin{scope}[xshift=0cm, yshift=\rowOneY, local bounding box=etaPlot]
						\def\scaleFactor{3.0}
						
						\draw[->, thick] (-0.5,0) -- (3.5,0) node[below] {$x$};
						\draw[->, thick] (0,-0.5) -- (0,3.5) node[left] {$\eta(x)$};
						
						\draw[very thick] (0,0) .. controls (0.5*\scaleFactor, 0) and (0.5*\scaleFactor, \scaleFactor) .. (\scaleFactor, \scaleFactor);
						
						\draw[dashed] (\scaleFactor,0) node[below] {$1$} -- (\scaleFactor,\scaleFactor) -- (0,\scaleFactor) node[left] {$1$};
						\node[below left] at (0,0) {$0$};
						\node[below=8pt, font=\bfseries] at (1.5, -0.5) {Step function $\eta(x)$};
					\end{scope}
					
					\begin{scope}[xshift=5.5cm, yshift=\rowOneY, local bounding box=xTildeOne]
						\def\lE{3.0}
						
						\draw[->, thick] (-0.5,0) -- (\lE+0.5,0) node[below] {$x$};
						\draw[->, thick] (0,-0.5) -- (0,\lE+0.5) node[left] {$\tilde{x}_e(x)$};
						
						\draw[very thick] (0,0) .. controls (0.5*\lE, 0) and (0.5*\lE, \lE) .. (\lE, \lE);
						
						\draw[dashed] (\lE,0) node[below] {$l_e$} -- (\lE,\lE) -- (0,\lE) node[left] {$l_e$};
						\node[below left] at (0,0) {$0$};
						\node[below=8pt, font=\bfseries] at (4.5, -0.5) {Coordinate transformations of edges with and without singularities};
					\end{scope}
					
					\begin{scope}[xshift=11cm, yshift=\rowOneY, local bounding box=xTildeTwo]
						\def\lE{3.0}
						\def\halfLE{1.5}
						\def\qE{2.2} 
						
						\draw[->, thick] (-0.5,0) -- (\lE+0.5,0) node[below] {$x$};
						\draw[->, thick] (0,-0.5) -- (0,\lE+0.5) node[left] {$\tilde{x}_e(x)$};
						
						\draw[very thick] 
						(0,0) .. controls (0.5*\halfLE, 0) and (0.5*\halfLE, \qE) .. (\halfLE, \qE)
						.. controls (\halfLE + 0.5*\halfLE, \qE) and (\lE - 0.5*\halfLE, \lE) .. (\lE, \lE);
						
						\draw[dashed] (\halfLE,0) node[below] {$\frac{l_e}{2}$} -- (\halfLE,\qE) -- (0,\qE) node[left] {$q_e$};
						\draw[dashed] (\lE,0) node[below] {$l_e$} -- (\lE,\lE) -- (0,\lE) node[left] {$l_e$};
						
						\node[below left] at (0,0) {$0$};
					\end{scope}

					\def\rowTwoY{-6cm} 
					
					\begin{scope}[xshift=0cm, yshift=\rowTwoY, local bounding box=dCaseOne]
						\def\lE{3.0}
						\def\dU{0.5}
						\def\dV{3.5}
						
						\draw[->, thick] (-0.5,0) -- (\lE+0.5,0) node[below] {$x$};
						\draw[->, thick] (0,-0.5) -- (0,\dV+0.5) node[left] {$\tilde{d}_e(x,x_0)$};
						
						\draw[very thick] (0,\dU) .. controls (0.5*\lE, \dU) and (0.5*\lE, \dV) .. (\lE, \dV);
						
						\draw[dashed] (0,\dU) node[left] {$d(i(e),x_0)$} -- (0,\dU); 
						\draw[dashed] (\lE,0) node[below] {$l_e$} -- (\lE,\dV) -- (0,\dV) node[left] {$d(j(e),x_0)$};
						
						\node[below left] at (0,0) {$0$};
						\node[below=8pt, font=\bfseries] at (1.5, -0.5) {Case 1};
					\end{scope}
					
					\begin{scope}[xshift=5.5cm, yshift=\rowTwoY, local bounding box=dCaseTwo]
						\def\lE{3.0}
						\def\dU{3.5}
						\def\dV{0.5}
						
						\draw[->, thick] (-0.5,0) -- (\lE+0.5,0) node[below] {$x$};
						\draw[->, thick] (0,-0.5) -- (0,\dU+0.5) node[left] {$\tilde{d}_e(x,x_0)$};
						
						\draw[very thick] (0,\dU) .. controls (0.5*\lE, \dU) and (0.5*\lE, \dV) .. (\lE, \dV);
						
						\draw[dashed] (\lE,0) node[below] {$l_e$} -- (\lE,\dV) -- (0,\dV) node[left] {$d(j(e),x_0)$};
						\node[left] at (0,\dU) {$d(i(e),x_0)$};
						\node[below left] at (0,0) {$0$};
						\node[below=8pt, font=\bfseries] at (1.5, -0.5) {Case 2};
					\end{scope}
					
					\begin{scope}[xshift=11cm, yshift=\rowTwoY, local bounding box=dCaseThree]
						\def\lE{3.0}
						\def\halfLE{1.5}
						\def\dU{1.0}
						\def\dV{1.8}
						\def\maxD{3.2} 
						
						\draw[->, thick] (-0.5,0) -- (\lE+0.5,0) node[below] {$x$};
						\draw[->, thick] (0,-0.5) -- (0,\maxD+0.5) node[left] {$\tilde{d}_e(x,x_0)$};
						
						\draw[very thick] 
						(0,\dU) .. controls (0.5*\halfLE, \dU) and (0.5*\halfLE, \maxD) .. (\halfLE, \maxD)
						.. controls (\halfLE + 0.5*\halfLE, \maxD) and (\lE - 0.5*\halfLE, \dV) .. (\lE, \dV);
						
						\draw[dashed] (\halfLE,0) node[below] {$\frac{l_e}{2}$} -- (\halfLE,\maxD) -- (0,\maxD) node[left] {$d(q_e,x_0)$};
						\draw[dashed] (\lE,0) node[below] {$l_e$} -- (\lE,\dV) -- (0,\dV) node[left] {$d(j(e),x_0)$};
						\node[left] at (0,\dU) {$d(i(e),x_0)$};
						\node[below left] at (0,0) {$0$};
						\node[below=8pt, font=\bfseries] at (1.5, -0.5) {Case 3};
					\end{scope}
					
				\end{tikzpicture}
				\caption{Illustration of step function, coordinate transformations and the modified distance function $\tilde{d}_e(x,x_0)$ for three cases previously shown by Fig~\ref{fig:distanceFunctions}.}
				\label{fig:smoothedDistance}
			\end{figure}
			
\newpage	
	\begin{proposition}\label{s4:P1}
		For any edge $e\in\mathcal{E}$, the following conclusions hold:
		\begin{enumerate}[(i)]
			\item For any $w\in\mathcal{V}$ with $w=i(e)$ or $w=j(e)$, the one-sided derivatives of $\tilde{d}_e$ at $w$ satisfy $\tilde{d}_e^\prime(w,x_0)=\tilde{d}_e^{\prime\prime}(w,x_0)=0$. For any edge $e$ containing a singular point, $\tilde{d}_e^\prime(s_e,x_0)=\tilde{d}_e^{\prime\prime}(s_e,x_0)=0$, where $s_e=\pi_e^{-1}(l_e/2)$ denotes the middle point of $e$. Moreover, $\tilde{d}\in C^2(\mathcal{G})$.
			\item There exists a constant $C>0$ such that for any $x\in(0,l_e)\setminus \{l_e/2\}$,
			\begin{equation}\label{s4:35}
				\le|\tilde{d}_e^\prime(x,x_0)\ri|\leq C,\quad \le|\tilde{d}_e^{\prime\prime}(x,x_0)\ri|\leq C.
			\end{equation}
		\end{enumerate}
	\end{proposition}
	\begin{proof}
		(i) We first claim that $d(\cdot,x_0)\big|_{\overline{I}_e}$ is $1$-Lipschitz continuous. For any $y,z\in\overline{I}_e$ on the same edge $e$, we know that the distance $d(\cdot,x_0)$ coincides with the arclength, i.e., $d(y,z) = |y-z|$.
		By the triangle inequality,
		$$
		d(y,x_0) \leq d(z,x_0) + d(y,z),\qquad
		d(z,x_0) \leq d(y,x_0) + d(y,z).
		$$
		Thus
		$$
		\big|d(y,x_0)-d(z,x_0)\big| \leq d(y,z) = |y-z|,
		$$
		so we confirm the claim. 
		
		For any $w\in\mathcal{V}$ and $e\ni w$, we next calculate the first-order right derivative of $\tilde{d}_e$ at the endpoint $w=i(e)$. By definition,
		$$
		\tilde{d}_e'(w^+,x_0)
		= \lim_{x\to 0^+} \frac{\tilde{d}_e(x,x_0)-\tilde{d}_e(0,x_0)}{x}
		= \lim_{x\to 0^+} \frac{d\big(\tilde{x}_e(x),x_0\big)-d\big(\tilde{x}_e(0),x_0\big)}{x}.
		$$
		By 1-Lipschitz continuity, we have
		$$
		\big|d\big(\tilde{x}_e(x),x_0\big)-d\big(\tilde{x}_e(0),x_0\big)\big|
		\leq \big|\tilde{x}_e(x)-\tilde{x}_e(0)\big| = |\tilde{x}_e(x)|.
		$$
		Hence,
		$$
		\left| \frac{\tilde{d}_e(x,x_0)-\tilde{d}_e(0,x_0)}{x} \right|
		\leq \frac{|\tilde{x}_e(x)|}{x}.
		$$
		If $e\cap\mathcal{V}_0=\emptyset$, then it follows from \eqref{s4:4} and \eqref{s4:33} that
		$$
		\lim_{x\to0^+}\frac{\tilde{x}_e(x)}{x}=\lim_{x\to0^+}\frac{l_e\eta\le(\frac{x}{l_e}\ri)}{x}=0
		$$
		by $\eta(0)=0$ and $\eta'_+(0)=0$. If $e\cap\mathcal{V}_0=q_e$, then on $[0,l_e/2]$,
		$\tilde{x}_e(x)=q_e\eta(2x/l_e)$, and similarly
		$$
		\lim_{x\to0^+}\frac{\tilde{x}_e(x)}{x}=\lim_{x\to0^+}\frac{q_e\eta\le(\frac{2x}{l_e}\ri)}{x}=0.
		$$
		By the squeeze theorem, we have $\tilde{d}_e'(w^+,x_0) = 0$. For the second-order right derivative, 
		$$
		\tilde{d}_e''(w^+,x_0)
		= \lim_{x\to0^+} \frac{\tilde{d}_e'(x,x_0)-\tilde{d}_e'(0^+,x_0)}{x}
		= \lim_{x\to0^+} \frac{\tilde{d}_e'(x,x_0)}{x}.
		$$
		For small $x>0$ and small $h$, 1-Lipschitz continuity gives
		$$
		\big|\tilde{d}_e(x+h,x_0)-\tilde{d}_e(x,x_0)\big|
		\leq \big|\tilde{x}_e(x+h)-\tilde{x}_e(x)\big|.
		$$
		Dividing by $h$ and letting $h\to0$ yields $\big|\tilde{d}_e'(x,x_0)\big| \leq \big|\tilde{x}_e'(x,x_0)\big|$. Thus
		$$
		\left| \frac{\tilde{d}_e'(x,x_0)}{x} \right|
		\leq \frac{|\tilde{x}_e'(x,x_0)|}{x}.
		$$
		Since $\tilde{x}_e'(x)=\eta'(x/l_e)$ or $\tilde{x}_e'(x)=(2q_e/l_e)\eta'(2x/l_e)$ and $\eta''_+(0)=0$, we obtain $\tilde{d}_e''(w^+,x_0) = 0$. The argument is symmetric at the  endpoint $w=j(e)$, and we omit it. Hence, the modified distance function has zero first (second)-order derivative at both endpoints (in one-sided sence). 
	
	   Next, we show the derivatives of the segment point $s_e$. Noting that $\eta'_-(1)=0$ and $\eta'_+(0)=0$, we deduce from \eqref{s4:34} that
	   $$
	   \lim_{x\to s_e^\pm}\frac{\tilde{x}_e(x)-\tilde{x}_e(s_e)}{x-s_e}=0.
	   $$
	   This, together with 
	   $$
	   \tilde{d}_e'(s_e^\pm,x_0)
	   = \lim_{x\to s_e^\pm} \frac{\tilde{d}_e(x,x_0)-\tilde{d}_e(s_e,x_0)}{x-s_e},
	   $$
	   and
	   $$
	   \big|\tilde{d}_e(x,x_0)-\tilde{d}_e(s_e,x_0)\big|
	   \leq \big|\tilde{x}_e(x)-\tilde{x}_e(s_e)\big|,
	   $$
	   implies that
	   $$
	   \tilde{d}_e'(s_e^-,x_0) = \tilde{d}_e'(s_e^+,x_0) = 0.
	   $$
	  Making use of  $\big|\tilde{d}_e'(x,x_0)\big|\leq\big|\tilde{x}_e'(x,x_0)\big|$, we have
	  $$
	  \left| \frac{\tilde{d}_e'(x,x_0)-\tilde{d}_e'(s_e,x_0)}{x-s_e} \right|
	  \leq \left| \frac{\tilde{x}_e'(x)-\tilde{x}_e'(s_e)}{x-s_e} \right|.
	  $$
	  It then follows from $\tilde{x}_e''(s_e^\pm)=0$ that $ \tilde{d}_e''(s_e,x_0) = 0$. Therefore, at every segment point $s_e$, the first and second derivatives agree and equal zero.
	  
	  Finally, on each edge $e\in\mathcal{E}$, $\tilde{x}_e\in C^2(\overline{I}_e)$ and $d$ is piecewise linear, so $\tilde{d}_e\in C^2(\overline{I}_e)$. At all vertices $w\in\mathcal{V}$ and segment points, the first and second derivatives vanish continuously across all adjacent edges. Thus,
	  $$
	  \tilde{d}\in C(\mathcal{G}),\quad \tilde{d}'\in C(\mathcal{G}),\quad \tilde{d}''\in C(\mathcal{G}),
	  $$
	  which implies
	  $$
	  \tilde{d}\in C^2(\mathcal{G}).
	  $$
		
		(ii) For any $e\in\mathcal{E}$ with $e\cap\mathcal{V}_0=\emptyset$, noting that $\inf_{e\in\mathcal{E}}l_e>r$, we obtain from \eqref{s4:33} that for any $x\in(0,l_e)$
		$$\le|\tilde{d}_e^\prime(x,x_0)\ri|
		= \le|d^\prime(\tilde{x}_e(x),x_0)\cdot \tilde{x}_e^\prime(x)\ri|=|\pm1|\cdot\le|\eta^\prime\le(\frac{x}{l_e}\ri)\ri|\leq C,$$ 
		and
		\begin{align*}
			\le|\tilde{d}_e^{\prime\prime}(x,x_0)\ri|=\le|d^{\prime\prime}(\tilde{x}_e(x),x_0)\cdot \le(\tilde{x}_e^\prime(x)\ri)^2+d^\prime(\tilde{x}_e(x),x_0)\cdot \tilde{x}_e^{\prime\prime}(x)\ri|=|\pm1|\cdot\le|\frac{1}{l_e}\eta^{\prime\prime}\le(\frac{x}{l_e}\ri)\ri|\leq \le|\frac{1}{r}\eta^{\prime\prime}\le(\frac{x}{l_e}\ri)\ri|\leq C,
		\end{align*}
		where  we have used $d^{\prime\prime}(\tilde{x}_e(x),x_0)=0$ and \eqref{s4:4}.
		
		For any $q_e\in\mathcal{V}_0$, as $q_e\in (0,l_e)$ and $\inf_{e\in\mathcal{E}}l_e>r$, we also have
		$$\le|\tilde{d}_e^\prime(x,x_0)\ri|
		=\le|\frac{2q_e}{l_e}\ri|\cdot\le|\eta^\prime\le(\frac{2x}{l_e}\ri)\ri|\leq C,\quad\le|\tilde{d}_e^{\prime\prime}(x,x_0)\ri|=\le|\frac{2q_e}{l_e}\ri|\cdot\le|\frac{2}{l_e}\eta^{\prime\prime}\le(\frac{2x}{l_e}\ri)\ri|\leq\le|\frac{4}{r}\eta^{\prime\prime}\le(\frac{2x}{l_e}\ri)\ri|\leq C,\quad\forall\ x\in\le(0,\frac{l_e}{2}\ri),$$
		and
		$$\le|\tilde{d}_e^\prime(x,x_0)\ri|
		=\le|-\frac{2(l_e-q_e)}{l_e}\ri|\cdot\le|\eta^\prime\le(\frac{2x-l_e}{l_e}\ri)\ri|\leq C,$$
		$$\le|\tilde{d}_e^{\prime\prime}(x,x_0)\ri|=\le|-\frac{2(l_e-q_e)}{l_e}\ri|\cdot\le|\frac{2}{l_e}\eta^{\prime\prime}\le(\frac{2x-l_e}{l_e}\ri)\ri|\leq\le|\frac{4}{r}\eta^{\prime\prime}\le(\frac{2x-l_e}{l_e}\ri)\ri|\leq C,\quad\forall\ x\in\le(\frac{l_e}{2},l_e\ri).$$
		This completes the proof.
	\end{proof}
	
	For every $e\in\mathcal{E}$, compared with the modified distance function $\tilde{d}_e=d\circ\tilde{x}_e$, we can also directly define the modified distance function to be the following forms.
	
	\begin{example}{(polynomial flat mollifier).} Define
			\begin{enumerate}[(i)]
			\item if $e \cap \mathcal{V}_0 = \emptyset$, and $d(x,x_0) = d\big(i(e),x_0\big) + x$ as $i(e)$ to $j(e)$,
			$$\tilde{d}_e(x,x_0)=d\big(i(e),x_0\big)+l_e\le(10\le(\frac{x}{l_e}\ri)^3-15\le(\frac{x}{l_e}\ri)^4+6\le(\frac{x}{l_e}\ri)^5\ri), \quad x\in[0,l_e];$$
				\item if $e \cap \mathcal{V}_0 = \emptyset$, and $d(x,x_0) = d\big(j(e),x_0\big) +l_e- x$ as $i(e)$ to $j(e)$,
			$$\tilde{d}_e(x,x_0)=d\big(j(e),x_0\big)+l_e-l_e\le(10\le(\frac{x}{l_e}\ri)^3-15\le(\frac{x}{l_e}\ri)^4+6\le(\frac{x}{l_e}\ri)^5\ri), \quad x\in[0,l_e];$$
			\item if $e \cap \mathcal{V}_0=\{q_e\}$, 
			\[
			\tilde{d}_e(x,x_0)=
			\begin{cases}
				d\big(i(e),x_0\big)+q_e\le(10\le(\frac{2x}{l_e}\ri)^3-15\le(\frac{2x}{l_e}\ri)^4+6\le(\frac{2x}{l_e}\ri)^5\ri), & x\in\le[0,\frac{l_e}{2}\ri],\\[1.5ex]
				d\big(i(e),x_0\big)+q_e+(q_e-l_e)\le(10\le(\frac{2x-l_e}{l_e}\ri)^3-15\le(\frac{2x-l_e}{l_e}\ri)^4+6\le(\frac{2x-l_e}{l_e}\ri)^5\ri), & x\in\le[\frac{l_e}{2},l_e\ri].\\[1.5ex]
			\end{cases}
			\]
		\end{enumerate}
	\end{example}
	
	\begin{example}{(infinite-order flat mollifier).}
		Define the standard smooth mollifier kernel
		\[
		\rho(t) =
		\begin{cases}
			e^{-\frac{1}{t(1-t)}}, & t\in(0,1),\\
			0, & t=0,\ 1.
		\end{cases}
		\]
		Let the normalization constant be
		\[
		C_0 = \int_0^1 \rho(\tau)\,d\tau,
		\]
		and define the exponential flat mollifier
		\[
		\zeta(t) = \frac{1}{C_0}\int_0^t \rho(\tau)\,d\tau.
		\]
		Then the left and right components of the modified distance function are given by
			\[
		\tilde{d}_e(x,x_0)=
		\begin{cases}
			d\big(i(e),x_0\big)+q_e\zeta\left(\frac{2x}{l_e}\right), & x\in\le[0,\frac{l_e}{2}\ri],\\[1.5ex]
			d\big(i(e),x_0\big)+q_e+(q_e-l_e)\zeta\left(\frac{2x-l_e}{l_e}\right), & x\in\le[\frac{l_e}{2},l_e\ri].\\[1.5ex]
		\end{cases}
\]
	\end{example}
	
	\begin{remark}
		Let step function $\tau:[0,1]\to[0,1]$ be a $C^2$-function satisfying
		\begin{equation*}
			\begin{cases}
				\tau(0)=0,\ \tau(1)=1,\ \tau^\prime_+(0)=\tau^\prime_-(1)=0;\\
				\tau^\prime(x)\geq0, \text{ for all}\ x\in (0, 1);\\
				\text{there exists}\ C > 0\ \text{such that}\ |\tau^\prime(x)|\leq C\ \text{and}\  |\tau^{\prime\prime}(x)|\leq C, \text{ for all}\ x\in (0, 1).
			\end{cases}
		\end{equation*}
		We may also define a $C^1$-modified distance function as 
		\begin{equation}\label{s4:5}
			\rho=\bigoplus_{e\in\mathcal{E}}\rho_e,\quad\rho_e(x,x_0):=d(\tilde{x}_e(x),x_0),
		\end{equation}
		where $\tilde{x}_e:\overline{I}_e\to\overline{I}_e$ is defined identically to \eqref{s4:33} and \eqref{s4:34}, with $\eta$ replaced by $\tau$. It is easy to check that this function is not in $C^2$, as its second derivative has a jump discontinuity at the segment points of edges containing singular points. Additionally, $\rho_e^\prime(w^\pm,x_0)=0$ for all $w\in\mathcal{V}$ and edges $e\ni w$, while $\rho_e^\prime(s_e,x_0)=0$ for any segment point $s_e$.
	\end{remark}

	\subsection{Nonexistence for nonnegative global solutions}
	
	Let $\phi\in C^2([0, \infty))$ be a cut-off function on $[0, \infty)$, which satisfies the following conditions
	\begin{equation}\label{s4:23}
		\begin{cases}
			\phi\equiv0\ \text{on}\ [2,\infty)\ \text{and}\ \phi\equiv1\ \text{on}\ [0,1]; \\
			\phi^\prime(x)\leq0,\ \text{for every}\ x\in [0, \infty); \\
			\text{there exists}\ C > 0\ \text{such that}\ |\phi^\prime(x)|\leq C\ \text{and}\  |\phi^{\prime\prime}(x)|\leq C, \text{ for all}\ x\in [0, \infty).
		\end{cases}
	\end{equation}
	Fix $R\geq R_0=\max\{2j,1\}$. We define 
	\begin{equation}\label{s4:37}
		\varphi(x):=\phi\le(\frac{\tilde{d}(x,x_0)}{R}\ri),\quad \forall\ x\in\mathcal{G},
	\end{equation}
	where $\tilde{d}$ is the modified distance function given in \eqref{s4:36}. It is clear that $\varphi$ is a compactly supported function defined on $\mathcal{G}$. We let $\text{supp}(\varphi)=\{x\in\mathcal{G}:\tilde{d}(x,x_0)<2R\}$ denote the support of $\varphi$. Since $\text{supp}(\varphi)$ intersects only finitely many edges, which we denote as  $\mathcal{E}_\varphi =\{\mathcal{G}_{e_1},  \mathcal{G}_{e_2}, \dots,  \mathcal{G}_{e_N}\}$, and its intersection with each such edge $e$ is a line segment $[a_e, b_e] \subset [0, l_e]$ with $0 \leq a_e < b_e \leq l_e$. Specifically, such edges can only take the following forms: 
	\begin{itemize}
		\item The entire edge contained in $\text{supp}(\varphi)$, i.e., $[a_e,b_e]=[0,l_e]$; 
		\item A portion of an edge is contained in $\text{supp}(\varphi)$, i.e., $[a_e,b_e]=[0,b_e]$ or $[a_e,b_e]=[a_e,l_e]$.
	\end{itemize}
For convenience, we use $[a_e,b_e]$ to denote all such intervals without specifying their exact forms. 

Let
\begin{equation}\label{s4:8}
	\mathcal{V}_1=\{x\in\mathcal{V}:d(x,x_0)\leq2R\},
\end{equation}
$$\mathcal{V}_2=\{x\in\mathcal{E}:\tilde{d}(x,x_0)=2R\}=\{a_e=\pi_e^{-1}(a_e):e\in\mathcal{E}_\varphi, a_e\not\in\mathcal{V} \}\cup\{b_e=\pi_e^{-1}(b_e):e\in\mathcal{E}_\varphi, b_e\not\in\mathcal{V} \}.$$
We then denote the set of all cut vertices by
$$\mathcal{V}_c=\p\mathcal{V}_1\cup\mathcal{V}_2,$$
and the set of all vertices in $\overline{\text{supp}(\varphi)}$ by
$$\mathcal{V}^\prime=\{\mathcal{V}_1\setminus\p\mathcal{V}_1\}\cup\mathcal{V}_c,\quad \{\mathcal{V}_1\setminus\p\mathcal{V}_1\}\cap\mathcal{V}_c=\emptyset.$$ 

We next present the following upper bounds.
\begin{lemma}\label{s4:L2}
	There exists a constant $C>0$ such that for each edge $e\in\mathcal{E}$,
	\begin{equation}\label{s4:1}
	|\varphi_e^{\prime\prime}(x)|\leq \frac{C}{R}\textbf{1}_{A_R}(x), \quad\forall \ x\in(0,l_e),
	\end{equation}
	and for every $x\in\mathcal{V}$,
	\begin{equation}\label{s4:6}
		-\Delta_{\mathcal{V}}\varphi(x)\leq \frac{C}{R}\textbf{1}_{D_R}(x),
	\end{equation}
	where $A_R=\le\{x\in\mathcal{G}:R\leq \tilde{d}(x,x_0)\leq 2R\ri\}$ and $D_R=\le\{x\in\mathcal{G}:R-j\leq d(x,x_0)\leq 2R+j\ri\}$. Moreover, for any $x\in\mathcal{V}^\prime$, 
	\begin{equation}\label{s2:1}
		[\mathcal{K}(\varphi)](x)=0.
	\end{equation}
\end{lemma}
\begin{proof}
	We start by analyzing the support of $\phi_e^\prime$ and $\phi_e^{\prime\prime}$. Observe that, for each edge $e\in\mathcal{E}$, $\phi_e^\prime\left( \frac{\tilde{d}(x,x_0)}{R} \right) \neq 0$ and $\phi_e^{\prime\prime}\left( \frac{\tilde{d}(x,x_0)}{R} \right) \neq 0$ if and only if $x\in A_R$, in which case $\le|\phi_e^\prime\le(\frac{\tilde{d}(x,x_0)}{R}\ri)\ri|\leq C$ and $\le|\phi_e^{\prime\prime}\le(\frac{\tilde{d}(x,x_0)}{R}\ri)\ri|\leq C$. By the chain rule and \eqref{s4:35}, we have for any $x\in(a_e,b_e)$
	\begin{equation*}
		|\varphi_e^{\prime\prime}(x)|=\le|\phi_e^{\prime\prime}\le(\frac{\tilde{d}(x,x_0)}{R}\ri)\frac{1}{R^2}\le(\tilde{d}_e^{\prime}(x,x_0)\ri)^2+\phi_e^\prime\le(\frac{\tilde{d}(x,x_0)}{R}\ri)\frac{1}{R}\tilde{d}_e^{\prime\prime}(x,x_0)\ri|\leq\frac{C}{R^2}\textbf{1}_{A_R}(x)+\frac{C}{R}\textbf{1}_{A_R}(x)\leq\frac{C}{R}\textbf{1}_{A_R}(x),
	\end{equation*}
	where we have used the fact that $R\geq1$.
	
	In our setting, since the conditions $(v)$ and $(vi)$ in \eqref{s3:1} are satisfied, $\tilde{d}$ automatically satisfies \eqref{s1:4} with $\alpha = 0$. In fact, for any $x\in\mathcal{V}$,
	\begin{align}\label{s4:25}
		\Delta_{\mathcal{V}} \tilde{d}(x,x_0)=\frac{1}{\mu_{\mathcal{V}}(x)}\sum_{y\sim x}\o(x,y)\le(d(y,x_0)-d(x,x_0)\ri)\leq \frac{1}{\mu_{\mathcal{V}}(x)}\sum_{y\sim x}\o(x,y)d(x,y)\leq \frac{1}{\mu_{\mathcal{V}}(x)}\sum_{y\sim x}\o(x,y)j\leq Cj.
	\end{align}
	By an argument analogous to that in [\cite{M-P-S}, Section 4], the estimate \eqref{s4:6} follows. This proof requires $R\geq \max\{2j,1\}$.  We omit the details as they overlap with existing proofs.
	
	Finally, if $x\in\mathcal{V}_1\setminus\p\mathcal{V}_1$, for every $e\ni x$ with $e\in\mathcal{E}_\varphi$, by $(i)$ of Proposition \ref{s4:P1}, we have $\tilde{d}_e^{\prime}(x,x_0)=0$. If $x\in \mathcal{V}_c$, then we derive that $\phi_e^\prime\le(\frac{\tilde{d}(x,x_0)}{R}\ri)=0$. It then follows that
	\begin{equation*}
		\varphi_e^{\prime}(x)=\phi_e^\prime\le(\frac{\tilde{d}(x,x_0)}{R}\ri)\frac{1}{R}\tilde{d}_e^{\prime}(x,x_0)=0,\quad\forall\ x\in \mathcal{V}^\prime,\ e\in\mathcal{E}_\varphi.
	\end{equation*}
	Since each vertex $x \in \mathcal{V}^\prime$ belongs to only finitely many edges, we deduce
	$$	[\mathcal{K}(\varphi)](x)=\sum_{e\ni x}\frac{d\varphi_e(x)}{dn}=0.$$
	Hence, this immediately implies the thesis. 
\end{proof}

On metric graphs, the integration by parts formula for $\varphi$, when using the standard distance $d$, gives rise to boundary terms encoding outer normal derivative information at the vertices, in contrast to the case on combinatorial graphs.  It is precisely the introduction of the modified distance function $\tilde{d}$ that causes such vertex contributions to vanish in our key lemma below.

\begin{lemma}\label{s2:L1}
	Let $s>\max\{2,\sigma/(\sigma-1)\}$ with $\sigma>1$. Suppose that $u\in \mathcal{D}(\mathcal{G})$ satisfies the condition \eqref{s2:19} and $\varphi:\mathcal{G}\to\mathbb{R}$ is defined in \eqref{s4:37}. Then we have
	\begin{equation}\label{s2:10}
		\mathcal{L}_{\Delta_{\mathcal{G}}u}(\varphi^s)= \int_\mathcal{G}u(\Delta_{\mathcal{V}}\varphi^s)d\mu_{\mathcal{V}}+\sum_{e\in{\mathcal{E}_\varphi}}\int_{a_e}^{b_e} u_e(x)(\varphi_e^s)^{\prime\prime}(x) dx,
	\end{equation}
	where $\mathcal{D}(\mathcal{G})$ and $\mathcal{L}_{\Delta_{\mathcal{G}}u}(\cdot)$ are specified in \eqref{s2:2} and \eqref{s2:15}, respectively.
\end{lemma}
\begin{proof}
	Let $\mathcal{E}_\varphi = \{ \mathcal{G}_{e_1},  \mathcal{G}_{e_2}, \dots,  \mathcal{G}_{e_N}\}$ denote the finite set of edges in \(\mathcal{G}\) that intersect $\text{supp}(\varphi)$. For each $e \in \mathcal{E}_\varphi$, let $a_e$ and $b_e$ be the endpoints of the segment $[a_e, b_e] \subset [0,l_e]$, which is exactly the intersection of $e$ with $\text{supp}(\varphi)$. In addition, $\varphi^s(x)\not=0$ if and only if $x\in \mathcal{V}_1\setminus\p\mathcal{V}_1$ or $x\in\bigcup_{e\in \mathcal{E}_\varphi}\{e\}\times(a_e,b_e)$. Recalling the definition of the linear functional $\mathcal{L}_{\Delta_{\mathcal{G}}u}$ in \eqref{s2:15}, we then proceed to estimate each term in the decomposition
	$$\mathcal{L}_{\Delta_{\mathcal{G}}u}(\varphi^s)=\int_\mathcal{G}\varphi^s(\Delta_{\mathcal{V}}u)d\mu_{\mathcal{V}}+\int_\mathcal{G}\varphi^s(\Delta_{\mathcal{E}}u)d\mu_{\mathcal{E}}.$$ 
	In view of \eqref{s2:18}, we have
\begin{equation}\label{s4:38}
	\int_\mathcal{G}\varphi^s(\Delta_{\mathcal{V}}u)d\mu_{\mathcal{V}}=\int_\mathcal{G}u(\Delta_{\mathcal{V}}\varphi^s)d\mu_{\mathcal{V}}.
\end{equation}
	By the formula for integration by parts twice, it follows from \eqref{s2:11} and \eqref{s2:12} that
	\begin{align}\label{s2:5}
		\nonumber	\int_\mathcal{G}\varphi^s(\Delta_{\mathcal{E}}u)d\mu_{\mathcal{E}}&=\sum_{e\in\mathcal{E}_\varphi}\int_{a_e}^{b_e} u_e^{\prime\prime}(x)\varphi^s_e(x) dx\\
		\nonumber&= -\sum_{e\in{\mathcal{E}_\varphi}}\int_{a_e}^{b_e} u_e^{\prime}(x)(\varphi_e^s)^{\prime}(x) dx+\sum_{e\in{\mathcal{E}_\varphi}}\le(u_e^\prime(b_e)\varphi_e^s(b_e)-u_e^\prime(a_e)\varphi^s_e(a_e)\ri)\\
		\nonumber&=\sum_{e\in{\mathcal{E}_\varphi}}\int_{a_e}^{b_e} u_e(x)(\varphi_e^s)^{\prime\prime}(x) dx-\sum_{e\in{\mathcal{E}_\varphi}}\le(u_e(b_e)(\varphi_e^s)^\prime(b_e)-u_e(a_e)(\varphi_e^s)^\prime(a_e)\ri)\\
		&\quad+\sum_{e\in{\mathcal{E}_\varphi}}\le(u_e^\prime(b_e)\varphi_e^s(b_e)-u_e^\prime(a_e)\varphi^s_e(a_e)\ri).
	\end{align}
	As every vertex $x \in \mathcal{V}^\prime$  has finite degree (i.e., is incident to only finitely many edges), we are able to transform the sum over edge endpoints into a sum over the adjacent edges of each vertex.  Thus, it follows from \eqref{s2:3}, \eqref{s2:4} and \eqref{s2:1} that 
	\begin{align*}
		&\sum_{e\in{\mathcal{E}_\varphi}}\le(u_e(b_e)(\varphi_e^s)^\prime(b_e)-u_e(a_e)(\varphi_e^s)^\prime(a_e)\ri)\\
		&=\sum_{e\in{\mathcal{E}_\varphi}}\le(su_e(b_e)\varphi_e^{s-1}(b_e)\varphi_e^\prime(b_e)-su_e(a_e)\varphi_e^{s-1}(a_e)\varphi_e^\prime(a_e)\ri)\\
		&=\sum_{e\in{\mathcal{E}_\varphi}}\le(su_e(b_e)\varphi_e^{s-1}(b_e)\frac{d\varphi_e(b_e)}{dn}+su_e(a_e)\varphi_e^{s-1}(a_e)\frac{d\varphi_e(a_e)}{dn}\ri)\\
		&=\sum_{x\in\mathcal{V}^\prime}su(x)\varphi^{s-1}(x)[\mathcal{K}(\varphi)](x)\\
		&=0,
	\end{align*}
where we have used the fact that $\mathcal{V}^\prime$ is a finite set. Similarly, since $[\mathcal{K}(u)](x)=0$ for any  $x\in \mathcal{V}_1\setminus\p\mathcal{V}_1$ and $\varphi\equiv0$ on $\mathcal{V}_c$, 
	\begin{align*}
		&\sum_{e\in{\mathcal{E}_\varphi}}\le(u_e^\prime(b_e)\varphi_e^s(b_e)-u_e^\prime(a_e)\varphi^s_e(a_e)\ri)\\
		&=\sum_{x\in \mathcal{V}_1\setminus\p\mathcal{V}_1}\varphi^{s}(x)[\mathcal{K}(u)](x)+\sum_{x\in\mathcal{V}_c}\varphi^{s}(x)[\mathcal{K}(u)](x)\\
		&=0.
	\end{align*} 
	Building on the above results combined with \eqref{s2:5}, we conclude that
	\begin{align}\label{s2:17}
	\int_\mathcal{G}\varphi^s(\Delta_{\mathcal{E}}u)d\mu_{\mathcal{E}}=\sum_{e\in{\mathcal{E}_\varphi}}\int_{a_e}^{b_e} u_e(x)(\varphi_e^s)^{\prime\prime}(x) dx.
	\end{align}
	Combining \eqref{s4:38} and \eqref{s2:17}, we get
	\begin{equation*}
		\mathcal{L}_{\Delta_{\mathcal{G}}u}(\varphi^s) =\int_\mathcal{G}u(\Delta_{\mathcal{V}}\varphi^s)d\mu_{\mathcal{V}}+\sum_{e\in{\mathcal{E}_\varphi}}\int_{a_e}^{b_e} u_e(x)(\varphi_e^s)^{\prime\prime}(x) dx.
	\end{equation*}
	This is the desired result. 
\end{proof}

\begin{lemma}\label{s4:L1}
	Let $u$ be a nonnegative solution to \eqref{s1:1} satisfying \eqref{s2:19}. If condition \eqref{s3:2} holds, then there exists a constant $C>0$ such that 
	$$\int_{\mathcal{G}}V(x)u^\sigma(x) d\mu_\mathcal{G}\leq C.$$
\end{lemma}
\begin{proof}
	Let $s>\max\{2,\sigma/(\sigma-1)\}$ with $\sigma>1$. Since $u$ fulfills \eqref{s2:21} and $0\leq\varphi\leq1$, we have 
	$$	V(x) u^{\sigma}(x)\varphi^s(x)\le(d\mu_{\mathcal{V}}+d\mu_{\mathcal{E}}\ri)\leq -\le(\Delta_\mathcal{G}  u(x)\ri) \varphi^s(x).$$
	Integrating both terms over $\mathcal{G}$, noting that \eqref{s2:22} and \eqref{s2:10}, we get
    \begin{align}\label{s4:3}
    \nonumber	\int_{\mathcal{G}}V(x) u^{\sigma}(x)\varphi^s(x)d\mu_{\mathcal{G}}&=\int_{\mathcal{G}}V(x) u^{\sigma}(x)\varphi^s(x)d\mu_{\mathcal{V}}+\int_{\mathcal{G}}V(x) u^{\sigma}(x)\varphi^s(x)d\mu_{\mathcal{E}}\\
    \nonumber	&\leq 	-\int_\mathcal{G}\le(\Delta_\mathcal{G}  u(x)\ri) \varphi^s(x)\\
    &= -\int_\mathcal{G}u(\Delta_{\mathcal{V}}\varphi^s)d\mu_{\mathcal{V}}-\sum_{e\in{\mathcal{E}_\varphi}}\int_{a_e}^{b_e} u_e(x)(\varphi_e^s)^{\prime\prime}(x) dx,
    \end{align}
   We now proceed to estimate each term separately.  Since
    $$\varphi^s(y)-\varphi^s(x)\geq s\varphi^{s-1}(x)\le(\varphi(y)-\varphi(x)\ri),\quad\forall\ x,y\in\mathcal{V},$$
    we obtain from \eqref{s4:6} that
    \begin{align}\label{s4:16}
    	\nonumber-\int_{\mathcal{G}}u\le(\Delta_\mathcal{V} \varphi^s \ri)d\mu_\mathcal{V}&=-\sum_{x\in\mathcal{V}}\sum_{y\sim x}\o(x,y)u(x)\le(\varphi^s(y)-\varphi^s(x)\ri)\\
    	\nonumber&\leq-s\sum_{x\in\mathcal{V}}\sum_{y\sim x}\o(x,y)u(x)\varphi^{s-1}(x)\le(\varphi(y)-\varphi(x)\ri)\\
    	\nonumber&=-s\sum_{x\in\mathcal{V}}\mu_\mathcal{V}(x)u(x)\varphi^{s-1}(x)\le(\frac{1}{\mu_\mathcal{V}(x)}\sum_{y\sim x}\o(x,y)\le(\varphi(y)-\varphi(x)\ri)\ri)\\
    	\nonumber&=-s\sum_{x\in\mathcal{V}}\mu_\mathcal{V}(x)u(x)\varphi^{s-1}(x)\Delta_\mathcal{V}\varphi(x)\\
    	&\leq \frac{C}{R}\sum_{x\in \mathcal{V}\cap D_R}\mu_\mathcal{V}(x)u(x)\varphi^{s-1}(x).
    \end{align}
    Using a standard application of Young’s inequality with exponent $\sigma>1$, we deduce that
    \begin{equation}\label{s4:12}
    	\frac{C}{R}\sum_{x\in \mathcal{V}\cap D_R}\mu_\mathcal{V}(x)u(x)\varphi^{s-1}(x)\leq \frac{1}{\sigma}\sum_{x\in \mathcal{V}\cap D_R}\mu_\mathcal{V}(x)V(x)u^\sigma(x)\varphi^{s}(x)+\frac{\sigma-1}{\sigma}\frac{C}{R^{\frac{\sigma}{\sigma-1}}}\sum_{x\in \mathcal{V}\cap D_R}\mu_\mathcal{V}(x)V^{-\frac{1}{\sigma-1}}(x)\varphi^{s-\frac{\sigma}{\sigma-1}}(x).
    \end{equation} 
    
    For the second term, in view of $s>2$, we have from \eqref{s4:1} that
    \begin{align}\label{s4:7}
    	\nonumber&-\sum_{e\in \mathcal{E}_\varphi}\int_{a_e}^{b_e}u_e(x)\le(\varphi_e^s\ri)^{\prime\prime}(x) dx\\
    	\nonumber&=-\sum_{e\in \mathcal{E}_\varphi}\int_{a_e}^{b_e}u_e(x)\le(s\varphi_e^{s-1}(x)\varphi_e^{\prime\prime}(x)+s(s-1)\varphi_e^{s-2}(x)\le(\varphi_e^\prime(x)\ri)^2\ri) dx\\
    	\nonumber&\leq -s\sum_{e\in \mathcal{E}_\varphi}\int_{a_e}^{b_e}u_e(x)\varphi_e^{s-1}(x)\varphi_e^{\prime\prime}(x)dx\\
    	\nonumber&\leq \frac{C}{R}\sum_{e\in \mathcal{E}}\int_{0}^{l_e}u_e(x)\varphi_e^{s-1}(x)\textbf{1}_{A_R}(x)dx\\
    &\leq \frac{1}{\sigma}\sum_{e\in \mathcal{E}}\int_{0}^{l_e}V_e(x)\varphi_e^s(x)u_e^\sigma(x)\textbf{1}_{A_R}(x) dx+\frac{\sigma-1}{\sigma}\frac{C}{R^{\frac{\sigma}{\sigma-1}}}\sum_{e\in \mathcal{E}}\int_{0}^{l_e}\varphi_e^{s-\frac{\sigma}{\sigma-1}}(x)V_e^{-\frac{1}{\sigma-1}}(x)\textbf{1}_{A_R}(x)dx.
    \end{align}	
	Then, substituting \eqref{s4:12} and \eqref{s4:7} into \eqref{s4:3}, we derive
	\begin{align*}
			&\int_{\mathcal{G}}V(x) u^{\sigma}(x)\varphi^s(x)d\mu_{\mathcal{G}}\\
			&\leq \frac{1}{\sigma}\sum_{x\in \mathcal{V}\cap D_R}\mu_\mathcal{V}(x)V(x)u^\sigma(x)\varphi^{s}(x)+\frac{1}{\sigma}\sum_{e\in \mathcal{E}}\int_{0}^{l_e}V_e(x)\varphi_e^s(x)u_e^\sigma(x)\textbf{1}_{A_R}(x) dx\\
			&\quad+\frac{\sigma-1}{\sigma}\frac{C}{R^{\frac{\sigma}{\sigma-1}}}\sum_{x\in \mathcal{V}\cap D_R}\mu_\mathcal{V}(x)V^{-\frac{1}{\sigma-1}}(x)\varphi^{s-\frac{\sigma}{\sigma-1}}(x)+\frac{\sigma-1}{\sigma}\frac{C}{R^{\frac{\sigma}{\sigma-1}}}\sum_{e\in \mathcal{E}}\int_{0}^{l_e}\varphi_e^{s-\frac{\sigma}{\sigma-1}}(x)V_e^{-\frac{1}{\sigma-1}}(x)\textbf{1}_{A_R}(x)dx\\
			&\leq \frac{1}{\sigma}\int_{\mathcal{G}}V(x) u^{\sigma}(x)\varphi^s(x)d\mu_{\mathcal{G}}+\frac{\sigma-1}{\sigma}\frac{C}{R^{\frac{\sigma}{\sigma-1}}}\sum_{x\in \mathcal{V}\cap D_R}\mu_\mathcal{V}(x)V^{-\frac{1}{\sigma-1}}(x)+\frac{\sigma-1}{\sigma}\frac{C}{R^{\frac{\sigma}{\sigma-1}}}\sum_{e\in \mathcal{E}\cap A_R}\int_{0}^{l_e}V_e^{-\frac{1}{\sigma-1}}(x)dx,
	\end{align*} 
	where we have used $0\leq\varphi^{s-\frac{\sigma}{\sigma-1}}\leq1$ due to $s>\sigma/(\sigma-1)$. Consequently,
	\begin{align}\label{s4:11}
	\int_{\mathcal{G}}V(x) u^{\sigma}(x)\varphi^s(x)d\mu_{\mathcal{G}}&\leq \frac{C}{R^{\frac{\sigma}{\sigma-1}}}\sum_{x\in \mathcal{V}\cap D_R}\mu_\mathcal{V}(x)V^{-\frac{1}{\sigma-1}}(x)+\frac{C}{R^{\frac{\sigma}{\sigma-1}}}\sum_{e\in \mathcal{E}\cap A_R}\int_{0}^{l_e}V_e^{-\frac{1}{\sigma-1}}(x)dx.
	\end{align}
	
	Next, we claim that $A_R\subset D_R$ for all $R\geq 2j$. In fact, if $\tilde{d}(x,x_0)\leq 2R$, by \eqref{s4:22}, we have
	$$d(x,x_0)\leq \tilde{d}(x,x_0)+j\leq 2R+j.$$
	On the other hand, suppose that $\tilde{d}(x,x_0)\geq R$, then 
	$$d(x,x_0)\geq \tilde{d}(x,x_0)-j\geq R-j.$$
	Hence, this concludes th proof of the claim. Recall $E_R=\{x\in\mathcal{G}:R\leq d(x,x_0)\leq 2R\}$. For every $R\geq2j$, it follows
	\begin{equation}\label{s4:10}
		A_R\subset D_R\subset \bigcup_{k=1}^3 E_{\frac{k}{2}R}.
	\end{equation}
	Note that $\varphi\geq0$ and $\varphi\equiv1$ on $\tilde{B}_R=\{x\in\mathcal{G}:\tilde{d}(x,x_0)<R\}$.  Then for every large enough $R$, by \eqref{s3:2} and \eqref{s4:11}, we have
	\begin{align*}
			\int_{\mathcal{G}}V(x) u^{\sigma}(x)\textbf{1}_{\tilde{B}_R}(x)d\mu_{\mathcal{G}}
			&\leq 	\int_{\mathcal{G}}V(x) u^{\sigma}(x)\varphi^s(x)d\mu_{\mathcal{G}}\\
			&\leq \frac{C}{R^{\frac{\sigma}{\sigma-1}}}\sum_{x\in \mathcal{V}\cap D_R}\mu_\mathcal{V}(x)V^{-\frac{1}{\sigma-1}}(x)+\frac{C}{R^{\frac{\sigma}{\sigma-1}}}\sum_{e\in \mathcal{E}\cap D_R}\int_{0}^{l_e}V_e^{-\frac{1}{\sigma-1}}(x)dx\\
			&\leq \frac{C}{R^{\frac{\sigma}{\sigma-1}}}\sum_{k=1}^3\sum_{x\in \mathcal{V}\cap E_{\frac{k}{2}R}}\mu_\mathcal{V}(x)V^{-\frac{1}{\sigma-1}}(x)+\frac{C}{R^{\frac{\sigma}{\sigma-1}}}\sum_{k=1}^3\sum_{e\in \mathcal{E}\cap E_{\frac{k}{2}R}}\int_{0}^{l_e}V_e^{-\frac{1}{\sigma-1}}(x)dx\\
			&\leq \frac{C}{R^{\frac{\sigma}{\sigma-1}}}\sum_{k=1}^3 \le(\frac{k}{2}R\ri)^{\frac{\sigma}{\sigma-1}}\\
			&\leq C,
	\end{align*}
	where the constant $C>0$ is independent of $R$. Letting $R\ra\infty$, we conclude that
	$$\int_{\mathcal{G}}V(x)u^\sigma(x) d\mu_\mathcal{G}\leq C.$$
	This is the desired result. 
\end{proof}

	With all the preceding lemmas established, we are ready to prove Theorem \ref{s3:T1}.\\

\noindent$\textbf{\emph{Proof of Theorem \ref{s3:T1}.}}$  Under the conditions of Lemma  \ref{s4:L1}, we aim to show that $u\equiv0$. Applying the H\"{o}lder inequality, in view of \eqref{s3:2} and \eqref{s4:16}, we get
\begin{align*}
	\nonumber-\int_{\mathcal{G}}u\le(\Delta_\mathcal{V} \varphi^s \ri)d\mu_\mathcal{V}&\leq  \frac{C}{R}\sum_{x\in \mathcal{V}\cap D_R}\mu_\mathcal{V}(x)u(x)\varphi^{s-1}(x)\\
		\nonumber&\leq \frac{C}{R}\le(\sum_{x\in \mathcal{V}\cap D_R}\mu_\mathcal{V}(x)V(x)u^\sigma(x)\varphi^{s}(x)\ri)^{\frac{1}{\sigma}}\le(\sum_{x\in \mathcal{V}\cap D_R}\mu_\mathcal{V}(x)V^{-\frac{1}{\sigma-1}}(x)\varphi^{s-\frac{\sigma}{\sigma-1}}(x)\ri)^{\frac{\sigma-1}{\sigma}}\\
	\nonumber&\leq \frac{C}{R}\le(\sum_{x\in \mathcal{V}\cap D_R}\mu_\mathcal{V}(x)V(x)u^\sigma(x)\ri)^{\frac{1}{\sigma}}\le(\sum_{x\in \mathcal{V}\cap D_R}\mu_\mathcal{V}(x)V^{-\frac{1}{\sigma-1}}(x)\ri)^{\frac{\sigma-1}{\sigma}}\\
	\nonumber&\leq \frac{C}{R}\le(\int_{D_R} V(x)u^\sigma(x)d\mu_{\mathcal{V}}\ri)^{\frac{1}{\sigma}}\le(\sum_{k=1}^3 \le(\frac{k}{2}R\ri)^{\frac{\sigma}{\sigma-1}}\ri)^{\frac{\sigma-1}{\sigma}}\\
	&\leq C\le(\int_{D_R} V(x)u^\sigma(x)d\mu_{\mathcal{V}}\ri)^{\frac{1}{\sigma}}.
\end{align*}
By \eqref{s3:2} , \eqref{s4:7} and \eqref{s4:10},  we have
\begin{align*}
	&-\sum_{e\in \mathcal{E}_\varphi}\int_{a_e}^{b_e}u_e(x)\le(\varphi_e^s\ri)^{\prime\prime}(x) dx\\
	&\leq \frac{C}{R}\sum_{e\in \mathcal{E}}\int_{0}^{l_e}u_e(x)\varphi_e^{s-1}(x)\textbf{1}_{D_R}(x)dx\\
	&\leq \frac{C}{R} \sum_{e\in\mathcal{E}}\le\{\le(\int_{0}^{l_e}V_e(x)\varphi_e^s(x)u_e^\sigma(x) \textbf{1}_{D_R}(x)dx\ri)^\frac{1}{\sigma}\le(\int_{a_e}^{b_e}\varphi_e^{s-\frac{\sigma}{\sigma-1}}(x)V_e^{-\frac{1}{\sigma-1}}(x)\textbf{1}_{D_R}(x)dx\ri)^\frac{\sigma-1}{\sigma}\ri\}\\\
	&\leq \frac{C}{R} \le(\sum_{e\in\mathcal{E}}\int_{0}^{l_e}V_e(x)\varphi_e^s(x)u_e^\sigma(x) \textbf{1}_{D_R}(x)dx\ri)^\frac{1}{\sigma}\le(\sum_{e\in\mathcal{E}}\int_{0}^{l_e}\varphi_e^{s-\frac{\sigma}{\sigma-1}}(x)V_e^{-\frac{1}{\sigma-1}}(x)\textbf{1}_{D_R}(x)dx\ri)^\frac{\sigma-1}{\sigma}\\
	&\leq \frac{C}{R} \le(\int_{D_R}V(x)u^\sigma(x) d\mu_\mathcal{E}\ri)^\frac{1}{\sigma}\le(\sum_{k=1}^3 \le(\frac{k}{2}R\ri)^{\frac{\sigma}{\sigma-1}}\ri)^{\frac{\sigma-1}{\sigma}}\\
	& \leq C\le(\int_{D_R}V(x)u^\sigma(x) d\mu_\mathcal{E}\ri)^\frac{1}{\sigma}.
\end{align*}
Combining \eqref{s4:3} and the above estimates, we deduce 
\begin{align*}
		\int_{\mathcal{G}}V(x) u^{\sigma}(x)\textbf{1}_{\tilde{B}_R}d\mu_{\mathcal{G}}&\leq 	\int_{\mathcal{G}}V(x) u^{\sigma}(x)\varphi^s(x)d\mu_{\mathcal{G}}\\
		&\leq C\le(\int_{D_R} V(x)u^\sigma(x)d\mu_{\mathcal{V}}\ri)^{\frac{1}{\sigma}}+C\le(\int_{D_R}V(x)u^\sigma(x) d\mu_\mathcal{E}\ri)^\frac{1}{\sigma}.
\end{align*}
By Lemma \ref{s4:L1}, the right-hand side tends to zero as $R \to \infty$. Therefore,
$$	\int_{\mathcal{G}}V(x)u^\sigma(x) d\mu_\mathcal{G}\leq 0.$$
Since $V$, $\mu_{\mathcal{V}}$ and $\mu_\mathcal{E}$ are positive and $u$ is nonnegative, 
\begin{equation*}
	\int_{\mathcal{G}}V(x)u^\sigma(x) d\mu_\mathcal{G}
	=\sum_{x\in\mathcal{V}}\mu_{\mathcal{V}}(x)V(x)u^\sigma(x)+\sum_{e\in\mathcal{E}}\int_0^{l_e}V_e(x)u_e^\sigma(x)dx\geq0.
\end{equation*}
We thus conclude that $u\equiv0$ on $\mathcal{G}$. $\hfill\Box$\\

	Using a modified $C^1$-distance function also yields the same conclusion, and we only provide a key sketch here. \\
	
\noindent$\textbf{\emph{Another proof of Theorem \ref{s3:T1}.}}$  Fix $x_0\in\mathcal{V}$ and $R\geq R_0$. We define another test function $\Phi:\mathcal{G}\to\mathbb{R}$ by
\begin{equation}\label{s4:2}
	\Phi(x):=\phi\le(\frac{\rho(x,x_0)}{R}\ri),
\end{equation}
where $\rho$ is given as in \eqref{s4:5} and $\phi$ is a cut-off function satisfying \eqref{s4:23}. Let $\mathcal{E}_{\Phi}$ be the set of edges intersected by the support $\text{supp}(\Phi)$. It is precisely the lack of second derivative information of the modified distance $\rho$ at segment points, the intersecting interval $[a_e,b_e]$ becomes complex. Specifically,
		\begin{itemize}
		\item The entire edge without a singular point is contained in $\text{supp}(\Phi)$, i.e., $[a_e,b_e]=[0,l_e]$; 
		\item A portion of an edge without a singular point is contained in $\text{supp}(\Phi)$, i.e., $[a_e,b_e]=[0,b_e]$ or $[a_e,b_e]=[a_e,l_e]$;
		\item The entire edge containing a singular point is contained in $\text{supp}(\Phi)$, i.e., $[a_e,b_e]=[0,s_e]$ or $[a_e,b_e]=[s_e,l_e]$;
		\item A portion of an edge containing a singular point is contained in $\text{supp}(\Phi)$, and $s_e\not\in\text{supp}(\varphi)$, i.e., $[a_e,b_e]=[0,b_e]$ or $[a_e,b_e]=[a_e,l_e]$ and $a_e,b_e\not=s_e$.
	\end{itemize}
	
	Let
	$$\mathcal{V}_3=\{x\in\mathcal{E}:\rho(x,x_0)=2R\},\quad\mathcal{V}_4=\{x\in\{s_e=l_e/2:e\cap\mathcal{V}_0=\emptyset,e\in\mathcal{E}\}:\rho(x,x_0)\leq2R\}.$$
	The set of all cut vertices and the set of all vertices in $\overline{\text{supp}(\Phi)}$ are defined by
	$$\mathcal{V}_c=\p\mathcal{V}_1\cup\mathcal{V}_3,\quad \mathcal{V}^\prime=\{\mathcal{V}_1\setminus\p\mathcal{V}_1\}\cup\{\mathcal{V}_4\setminus \mathcal{V}_3\}\cup\mathcal{V}_c,$$ 
	where $V_1$ is given in \eqref{s4:8}. Clearly, 
	$$\{\mathcal{V}_1\setminus\p\mathcal{V}_1\}\cap\{\mathcal{V}_4\setminus \mathcal{V}_3\}=\{\mathcal{V}_1\setminus\p\mathcal{V}_1\}\cap\mathcal{V}_c=\{\mathcal{V}_4\setminus \mathcal{V}_3\}\cap\mathcal{V}_c=\emptyset.$$ 
	
	Since the modified distance functions $\tilde{d}(x,x_0)$ and $\rho(x,x_0)$ coincide with the original distance function $d(x,x_0)$ at vertices $x\in\mathcal{V}$, the estimate 
	\begin{equation*}
		-\Delta_{\mathcal{V}}\Phi(x)\leq \frac{C}{R}\textbf{1}_{D_R}(x),\quad\forall\ x\in\mathcal{V},
	\end{equation*}
	 remain consistent with the previous proof and require no changes. We thus only focus on the differences hereafter, namely the estimates on edges. As $(ii)$ in Proposition \ref{s4:P1} still holds for $\rho$, 
	 $$	|\Phi_e^{\prime\prime}(x)|\leq \frac{C}{R}\textbf{1}_{A_R}(x), \quad\forall \ e\in\mathcal{E},\ x\in(a_e,b_e),$$
	 follows similarly, where $A_R=\le\{x\in\mathcal{E}:R\leq \rho(x,x_0)\leq 2R\ri\}$.  
	 
	 For $x\in\mathcal{V}_4\setminus \mathcal{V}_3$, it is obvious that $x$ is a segment point $s_e$ on edge $e$ of degree $2$, incident to exactly two edges $e_1$ and $e_2$ that can be represented by intervals $(0,l_e/2)$ and $(l_e/2,l_e)$ respectively. Thus, if $u\in\mathcal{D}(\mathcal{G})$ is a nonnegative solution of \eqref{s1:1}, we have $u_e\in C^1(\overline{I}_e)$, so $u_{e_1}^\prime(s_e)=u_{e_2}^\prime(s_e)$. This implies that 
	 \begin{align*}
	 	\sum_{x\in \mathcal{V}_4\setminus \mathcal{V}_3}\Phi^{s}(x)[\mathcal{K}(u)](s_e)
	 	&=\sum_{x\in\mathcal{V}_4\setminus \mathcal{V}_3}\Phi^{s}(x)\le(\frac{d\Phi_{e_1}(j(e_1))}{dn}+\frac{d\Phi_{e_2}(i(e_2))}{dn}\ri)\\
	 	&=\sum_{x\in \mathcal{V}_4\setminus \mathcal{V}_3}\Phi^{s}(x)\le(u_{e_1}^\prime(s_e)-u_{e_2}^\prime(s_e)\ri)\\
	 	&=0.
	 \end{align*}
	Since $\rho_e^\prime(w^\pm,x_0)=0$ for any $w\in\mathcal{V}$ and all $e\in \mathcal{E}_\Phi$, $\rho_e^\prime(s_e,x_0)=0$ for every segment point $s_e$, we also obtain that
	$$[\mathcal{K}(\Phi)](x)=0,\quad\forall\ x\in\mathcal{V}^\prime.$$ 
	 Multiply both sides of \eqref{s2:21} by the test function $\Phi^s$, and integration leads to
	   \begin{align*}
	 	\nonumber	\int_{\mathcal{G}}V(x) u^{\sigma}(x)\Phi^s(x)d\mu_{\mathcal{G}}&\leq -\int_\mathcal{G}u(\Delta_{\mathcal{V}}\Phi^s)d\mu_{\mathcal{V}}-\sum_{e\in{\mathcal{E}_\Phi}}\int_{a_e}^{b_e} u_e(x)(\Phi_e^s)^{\prime\prime}(x) dx+\sum_{x\in\mathcal{V}^\prime}su(x)\varphi^{s-1}(x)[\mathcal{K}(\Phi)](x)\\
	 	&\quad-\sum_{x\in \mathcal{V}_1\setminus\p\mathcal{V}_1}\Phi^{s}(x)[\mathcal{K}(u)](x)-\sum_{x\in\mathcal{V}_4\setminus \mathcal{V}_3}\Phi^{s}(x)[\mathcal{K}(u)](x)-\sum_{x\in\mathcal{V}_c}\Phi^{s}(x)[\mathcal{K}(u)](x)\\
	 	&=-\int_\mathcal{G}u(\Delta_{\mathcal{V}}\Phi^s)d\mu_{\mathcal{V}}-\sum_{e\in{\mathcal{E}_\Phi}}\int_{a_e}^{b_e} u_e(x)(\Phi_e^s)^{\prime\prime}(x) dx,
	 \end{align*}
	where we have used $u$ satisfies the condition \eqref{s2:19} and $\Phi(x)=0$ for each vertex $x\in \mathcal{V}_c$. The rest of the proof follows similarly to the previous one and is not detailed herein.$\hfill\Box$

	\subsection{Nonexistence for sign-changing global solutions}
	
		Fix $x_0\in\mathcal{V}$ and let $R\geq R_0$. Let $\psi:[-j/R, \infty)\to(0,\infty)$ be a function satisfying 
	\begin{itemize}
		\item $\psi\in C^2([-j/R, \infty))$;
		\item $\psi\equiv1$ on $[-j/R,1]$ and $\psi(r)=e^{-\delta r}$ on $[2,\infty)$ for some  $\delta>0$;
		\item $\psi^\prime(x)\leq0$ for all $x\in [-j/R, \infty)$.
	\end{itemize}
	Then there exist constants $C_1>0$, $C_2 > 0$ such that
	\begin{equation}\label{s4:13}
		0<C_2e^{-\delta r}\leq\psi(r)\leq C_1e^{-\delta r},\quad\forall\ r\in \le[-\frac{j}{R}, \infty\ri),
	\end{equation}
	and
	\begin{equation}\label{s4:14}
		|\psi^\prime(r)|\leq C_1 e^{-\delta r},\quad|\psi^{\prime\prime}(r)|\leq C_1e^{-\delta r},\quad\forall\ r\in \le[-\frac{j}{R}, \infty\ri).
	\end{equation}
	Define 
	\begin{equation}\label{s4:15}
		\Psi(x):=\psi\le(\frac{\tilde{d}(x,x_0)-j}{R}\ri),\quad \forall\ x\in\mathcal{G}.
	\end{equation}
	Here, $\tilde{d}$ is a modified distance function defined in \eqref{s4:36}, and $\Psi>0$ has support on the entire graph $\mathcal{G}$. We now establish estimates for the derivatives of $\Psi$.
	\begin{lemma}\label{s4:L3}
		There exists a constant $C>0$ such that for each edge $e\in\mathcal{E}$,
		\begin{equation}\label{s4:9}
		|\Psi_e^{\prime\prime}(x)|\leq \frac{C}{R}e^{\frac{-\delta \tilde{d}(x,x_0)}{R}}\textbf{1}_{\tilde{B}_{R+j}^c}(x), \quad\forall \ x\in(0,l_e),
		\end{equation}
		and for all $x\in\mathcal{V}$,
		\begin{equation}\label{s4:26}
			\le|\Delta_{\mathcal{V}}\Psi(x)\ri|\leq\frac{C}{R}e^{-\frac{\delta d(x,x_0)}{R}}\textbf{1}_{B_{R}^c}(x),
		\end{equation}
		where $\tilde{B}_{R+j}^c=\le\{x\in\mathcal{G}:\tilde{d}(x,x_0)\geq R+j\ri\}$ and $B_{R}^c=\le\{x\in\mathcal{G}:d(x,x_0)\geq R\ri\}$.  Moreover, for any $x\in\mathcal{V}$, 
\begin{equation}\label{s4:17}
	[\mathcal{K}(\Psi)](x)=0.
\end{equation}
	\end{lemma}
	
	\begin{proof}
		For each $e\in\mathcal{E}$, by Proposition \ref{s4:P1}, we obtain
		\begin{align*}
			\le|\Psi_e^{\prime\prime}(x)\ri|&=\le|\psi_e^{\prime\prime}\le(\frac{\tilde{d}(x,x_0)-j}{R}\ri)\frac{1}{R^2}\le(\tilde{d}_e^{\prime}(x,x_0)\ri)^2+\psi_e^\prime\le(\frac{\tilde{d}(x,x_0)-j}{R}\ri)\frac{1}{R}\tilde{d}_e^{\prime\prime}(x,x_0)\ri|\\
			&\leq\frac{C}{R^2}e^{-\delta\frac{\tilde{d}(x,x_0)-j}{R}}\textbf{1}_{\tilde{B}_{R+j}^c}(x)+\frac{C}{R}e^{-\delta\frac{\tilde{d}(x,x_0)-j}{R}}\textbf{1}_{\tilde{B}_{R+j}^c}(x)\\
			&\leq \frac{C}{R}e^{\frac{-\delta \tilde{d}(x,x_0)}{R}}\textbf{1}_{\tilde{B}_{R+j}^c}(x), \quad\forall\ x\in(0,l_e),
		\end{align*}
		where we have used the fact that $\psi_e^{\prime}(x)=0$, $\psi_e^{\prime\prime}(x)=0$ on $\tilde{B}_{R+j}$, and $R\geq R_0=\max\{2j,1\}$.
		
		For the vertex-based Laplacian, note that for each $x\in\mathcal{V}$, 
		\begin{align*}
			\Delta_{\mathcal{V}}\Psi(x)&=\frac{1}{\mu_{\mathcal{V}}(x)}\sum_{y\sim x}\o(x,y)\le(\Psi(y)-\Psi(x)\ri)\\
			&=\frac{1}{\mu_{\mathcal{V}}(x)}\sum_{y\sim x}\o(x,y)\le(\psi\le(\frac{\tilde{d}(y,x_0)-j}{R}\ri)-\psi\le(\frac{\tilde{d}(x,x_0)-j}{R}\ri)\ri)\\
			&=\frac{1}{\mu_{\mathcal{V}}(x)}\sum_{y\sim x}\o(x,y)\le(\psi\le(\frac{d(y,x_0)-j}{R}\ri)-\psi\le(\frac{d(x,x_0)-j}{R}\ri)\ri)\\
			&=\frac{1}{\mu_{\mathcal{V}}(x)}\sum_{y\sim x}\o(x,y)\psi^\prime\le(\frac{d(x,x_0)-j}{R}\ri)\le(\frac{d(y,x_0)-d(x,x_0)}{R}\ri)\\
			&\quad+\frac{1}{2\mu_{\mathcal{V}}(x)}\sum_{y\sim x}\o(x,y)\psi^{\prime\prime}\le(\xi\ri)\le(\frac{d(y,x_0)-d(x,x_0)}{R}\ri)^2,
		\end{align*}
		for some $\xi$ between $(d(y,x_0)-j)/R$ and $(d(x,x_0)-j)/R$. Since $\psi^\prime\equiv0$ on $\tilde{B}_{R+j}$ and $\xi\leq1$ for all $x\in B_R\cap\mathcal{V}$, we have $\Delta_{\mathcal{V}}\Psi(x)\equiv0$ for $x\in B_R\cap\mathcal{V}$. For $x\in B_R^c\cap \mathcal{V}$ and $R\geq 2j$, observe that 
		\begin{equation}\label{s4:24}
			\xi\geq \min\le\{\frac{d(y,x_0)-j}{R},\frac{d(x,x_0)-j}{R}\ri\}\geq \frac{d(x,x_0)-2j}{R}\geq\frac{d(x,x_0)}{R}-1\geq0.
		\end{equation}
		By \eqref{s4:25}, \eqref{s4:14} and \eqref{s4:24}, we get for any $x\in\mathcal{V}$
		\begin{align*}
			\le|\Delta_{\mathcal{V}}\Psi(x)\ri|&\leq \frac{1}{R}\le|\psi^\prime\le(\frac{d(x,x_0)-j}{R}\ri)\ri|\frac{1}{\mu_{\mathcal{V}}(x)}\sum_{y\sim x}\o(x,y)\le(d(y,x_0)-d(x,x_0)\ri)\textbf{1}_{B_{R}^c}(x)\\
			&\quad+\frac{C}{2R^2}\frac{1}{\mu_{\mathcal{V}}(x)}\sum_{y\sim x}\o(x,y)e^{-\frac{\delta d(x,x_0)}{R}}\le(d(y,x_0)-d(x,x_0)\ri)^2\textbf{1}_{B_{R}^c}(x)\\
			&\leq \frac{C}{R} e^{-\delta\frac{d(x,x_0)-j}{R}} \Delta_\mathcal{V}d(x,x_0)\textbf{1}_{B_{R}^c}(x)+\frac{C}{2R^2}e^{-\frac{\delta d(x,x_0)}{R}}\textbf{1}_{B_{R}^c}(x)\frac{1}{\mu_{\mathcal{V}}(x)}\sum_{y\sim x}\o(x,y)j^2\\
			&\leq\le(\frac{C}{R}e^{-\delta\frac{ d(x,x_0)}{R}}+\frac{C }{2R^2}e^{-\delta\frac{d(x,x_0)}{R}}\ri)\textbf{1}_{B_{R}^c}(x)\\
			&\leq\frac{C}{R}e^{-\frac{\delta d(x,x_0)}{R}}\textbf{1}_{B_{R}^c}(x),
		\end{align*}
		where we have used \eqref{s3:1}-(v) and 
		$$|d(y,x_0)-d(x,x_0)|\leq d(x,y)\leq j.$$  
		
		Finally, for all $x\in\mathcal{V}$, for every $e\ni x$ with $e\in\mathcal{E}$, it follows from $(i)$ of Proposition \ref{s4:P1} that $\tilde{d}_e^{\prime}(x,x_0)=0$. By chain rule, we have
		\begin{equation*}
			\Psi_e^\prime(x)=\psi_e^\prime\le(\frac{\tilde{d}(x,x_0)-j}{R}\ri)\frac{1}{R}\tilde{d}_e^\prime(x,x_0)=0,\quad\forall\ x\in \mathcal{V},\ e\in\mathcal{E}.
		\end{equation*}
		Thus,
		$$	[\mathcal{K}(\Psi)](x)=\sum_{e\ni x}\frac{d\Psi_e(x)}{dn}=0,$$
		which completes the proof.
	\end{proof} 
	
		For the function \eqref{s4:15}, which lacks compact support, we need to strengthen the conditions on $u$; however, we can still derive the integration by parts formula for this case.
	\begin{lemma}\label{s4:L4}
		Suppose that $u$ is a solution to \eqref{s1:1} satisfying \eqref{s2:19}, and let $\Psi$ be as in \eqref{s4:15}. Assume $u\in X_\alpha$ with $\alpha=\delta/R>0$. Then we have
		\begin{equation}\label{s4:18}
	\int_\mathcal{G}\Psi(\Delta_{\mathcal{V}}u)d\mu_{\mathcal{V}}+\int_\mathcal{G}\Psi(\Delta_{\mathcal{E}}u)d\mu_{\mathcal{E}}=	\int_\mathcal{G}u(\Delta_{\mathcal{V}}\Psi)d\mu_{\mathcal{V}}+\int_\mathcal{G}u(\Delta_{\mathcal{E}}\Psi)d\mu_{\mathcal{E}}.
		\end{equation}
	\end{lemma}
	\begin{proof}
		Since $u\in X_\alpha$ and $\alpha=\delta/R$ for some $\delta>0$ and $R\geq\max\{1,2j\}$, by \eqref{s2:13}, we know 
		\begin{equation}\label{s4:27}
			\sum_{x\in\mathcal{V}}\mu_{\mathcal{V}}(x)|u(x)| e^{-\delta\frac{d(x,x_0)}{R}}\leq C,\quad\sum_{e\in\mathcal{E}}\int_{0}^{l_e}|u_e(x)|e^{-\delta \frac{d(x,x_0)}{R}}dx\leq C.
		\end{equation}
		It follows from \eqref{s4:22} that for any $R\geq 2j$,
		\begin{equation*}
			C\geq \sum_{e\in\mathcal{E}}\int_{0}^{l_e}|u_e(x)|e^{-\delta \frac{d(x,x_0)}{R}}dx\geq \sum_{e\in\mathcal{E}}\int_{0}^{l_e}|u_e(x)|e^{-\delta \frac{\tilde{d}(x,x_0)+j}{R}}dx\geq \sum_{e\in\mathcal{E}}\int_{0}^{l_e}|u_e(x)|e^{-\delta \frac{\tilde{d}(x,x_0)}{R}}e^{-\frac{\delta}{2}}dx.
		\end{equation*}
		This implies that
		\begin{equation}\label{s4:39}
			\sum_{e\in\mathcal{E}}\int_{0}^{l_e}|u_e(x)|e^{-\delta \frac{\tilde{d}(x,x_0)}{R}}dx\leq C.
		\end{equation}
		In view of \eqref{s4:13}, we have
		$$|\Psi(x)|\leq Ce^{-\delta\frac{d(x,x_0)}{R}},\quad\forall\ x\in\mathcal{V}.$$
	  By Proposition 5.2 in \cite{M-P-S2}, combined with \eqref{s4:27}, we deduce that
	  \begin{equation}\label{s4:19}
	  		\int_\mathcal{G}\Psi(\Delta_{\mathcal{V}}u)d\mu_{\mathcal{V}}=\int_\mathcal{G}u(\Delta_{\mathcal{V}}\Psi)d\mu_{\mathcal{V}},
	  \end{equation}
	  which is omitted here. Hence, in order to establish \eqref{s4:18}, it is enough to verify
	  \begin{equation}\label{s4:20}
	  	\int_\mathcal{G}\Psi(\Delta_{\mathcal{E}}u)d\mu_{\mathcal{E}}=	\int_\mathcal{G}u(\Delta_{\mathcal{E}}\Psi)d\mu_{\mathcal{E}}.
	  \end{equation}
	  By the formula for integration by parts twice, it follows from \eqref{s2:19} and \eqref{s4:17}that
	  \begin{align}\label{s4:21}
	  	\nonumber	\int_\mathcal{G}\Psi(\Delta_{\mathcal{E}}u)d\mu_{\mathcal{E}}&=\sum_{e\in\mathcal{E}}\int_{0}^{l_e} u_e^{\prime\prime}(x)\Psi_e(x) dx\\
	  	\nonumber&= -\sum_{e\in{\mathcal{E}}}\int_{0}^{l_e} u_e^{\prime}(x)\Psi_e^{\prime}(x) dx+\sum_{e\in{\mathcal{E}}}\le(u_e^\prime(l_e)\Psi_e(l_e)-u_e^\prime(0)\Psi_e(0)\ri)\\
	  	\nonumber&=\sum_{e\in{\mathcal{E}}}\int_{0}^{l_e} u_e(x)\Psi_e^{\prime\prime}(x) dx-\sum_{e\in{\mathcal{E}}}\le(u_e(l_e)(\Psi_e^\prime(l_e)-u_e(0)\Psi_e^\prime(0)\ri)+\sum_{e\in{\mathcal{E}}}\le(u_e^\prime(l_e)\Psi_e(l_e)-u_e^\prime(0)\Psi_e(0)\ri)\\
	  	\nonumber&=\sum_{e\in{\mathcal{E}}}\int_{0}^{l_e} u_e(x)\Psi_e^{\prime\prime}(x) dx-\sum_{x\in\mathcal{V}}u(x)[\mathcal{K}(\Psi)](x)+\sum_{x\in\mathcal{V}}\Psi(x)[\mathcal{K}(u)](x)\\
	  	&=\sum_{e\in{\mathcal{E}}}\int_{0}^{l_e} u_e(x)\Psi_e^{\prime\prime}(x) dx.
	  \end{align}
	  Moreover, by \eqref{s4:39} and \eqref{s4:21}, we derive from \eqref{s4:9} that
	  \begin{align*}
	  	\le|\int_\mathcal{G}\Psi(\Delta_{\mathcal{E}}u)d\mu_{\mathcal{E}}\ri|\leq\sum_{e\in{\mathcal{E}}}\int_{0}^{l_e} |u_e(x)|\le|\Psi_e^{\prime\prime}(x) \ri|dx\leq \frac{C}{R_0} \sum_{e\in\mathcal{E}}\int_{0}^{l_e}|u_e(x)|e^{-\delta \frac{\tilde{d}(x,x_0)}{R}}dx\leq C,
	  \end{align*}
	  which yields that 
	  $$\int_\mathcal{G}\Psi(\Delta_{\mathcal{E}}u)d\mu_{\mathcal{E}}$$
	  converges absolutely. Thus, \eqref{s4:20} follows immediately.
	\end{proof}
	
		Building on Lemma \ref{s4:L4}, we can now show a priori $L_{V}^\sigma$ estimates for the solution $u\in X_\alpha$.
	\begin{lemma}\label{s4:L5}
		Under the conditions of Lemmas \ref{s4:L3} and \ref{s4:L4},  assume further that the potential $V$ satisfies \eqref{s3:4}. Then $u\in L_{V}^\sigma(\mathcal{G})$ for all $\sigma>1$. Specifically, there exists a constant $C>0$ such that 
		$$\int_{\mathcal{G}}V(x)|u(x)|^\sigma d\mu_\mathcal{G}\leq C.$$
	\end{lemma}
	
	\begin{proof}
		Observe that $u\in \mathcal{D}(\mathcal{G})$ is a solution of \eqref{s1:1} and $0<\Psi(x)\leq1$ for all $x\in\mathcal{G}$. Multiplying both sides of \eqref{s2:21} by $\Psi(x)$,  we have
		$$	V(x)\Psi(x) |u(x)|^{\sigma}\le(d\mu_{\mathcal{V}}+d\mu_{\mathcal{E}}\ri)\leq-\le(\Delta_\mathcal{G} u(x)\ri)\Psi(x).$$
		Integrating over $\mathcal{G}$, we obtain from \eqref{s4:18} that
		\begin{align}\label{s4:29}
			\nonumber\int_{\mathcal{G}}V(x)\Psi(x) |u(x)|^{\sigma} d\mu_{\mathcal{G}}&=\int_{\mathcal{G}}V(x)\Psi(x) |u(x)|^{\sigma}d\mu_{\mathcal{V}}+\int_{\mathcal{G}}V(x)\Psi(x) |u(x)|^{\sigma}d\mu_{\mathcal{E}}\\
			&\leq-	\int_\mathcal{G}u(\Delta_{\mathcal{V}}\Psi)d\mu_{\mathcal{V}}-\int_\mathcal{G}u(\Delta_{\mathcal{E}}\Psi)d\mu_{\mathcal{E}}.
		\end{align}
		Using \eqref{s3:4}, \eqref{s4:26}, and Young's inequality with exponent $\sigma>1$, we derive that for any $\epsilon>0$,
		\begin{align}\label{s4:30}
			\nonumber\le|-\int_\mathcal{G}u(\Delta_{\mathcal{V}}\Psi)d\mu_{\mathcal{V}}\ri|&\leq \sum_{x\in\mathcal{V}}\mu_{\mathcal{V}}(x) |u(x)|\le|\Delta_\mathcal{V} \Psi(x) \ri|\\
			\nonumber&\leq\frac{C}{R}\sum_{x\in\mathcal{V}}\mu_{\mathcal{V}}(x) |u(x)|e^{-\frac{\delta d(x,x_0)}{R}}\textbf{1}_{B_{R}^c}(x)\\
			\nonumber&\leq\epsilon\sum_{x\in\mathcal{V}}\mu_{\mathcal{V}}(x) |u(x)|^{\sigma}V(x)e^{-\frac{\delta d(x,x_0)}{R}}\textbf{1}_{B_{R}^c}(x)+\frac{C_{\epsilon}}{R^{\frac{\sigma}{\sigma-1}}}\sum_{x\in\mathcal{V}}\mu_{\mathcal{V}}(x) V^{-\frac{1}{\sigma-1}}(x)e^{-\frac{\delta d(x,x_0)}{R}}\textbf{1}_{B_{R}^c}(x)\\
			\nonumber&\leq \epsilon C\sum_{x\in\mathcal{V}}\mu_{\mathcal{V}}(x) |u(x)|^{\sigma}V(x)\Psi(x)+\frac{C_{\epsilon}}{R^{\frac{\sigma}{\sigma-1}}}\sum_{x\in\mathcal{V}\cap B_{R}^c}\mu_{\mathcal{V}}(x) V^{-\frac{1}{\sigma-1}}(x)e^{-\frac{\delta d(x,x_0)}{R}}\\
			&\leq \epsilon C\sum_{x\in\mathcal{V}}\mu_{\mathcal{V}}(x) |u(x)|^{\sigma}V(x)\Psi(x)+C_\epsilon.
		\end{align}
		where in the penultimate step we have used \eqref{s4:13}. For the edge term, it is easy to obtain
		$$\tilde{B}_{R+j}^c\subset B_{R}^c.$$
		In fact, if $x\in \tilde{B}_{R+j}^c$, then $\tilde{d}(x,x_0)\geq R+j$. It follows from \eqref{s4:22} that 
		$$d(x,x_0)\geq \tilde{d}(x,x_0)-j\geq R.$$ 
		Thus, $x\in B_{R}^c$. By means of \eqref{s3:4} and \eqref{s4:9}, for any $\epsilon>0$, we deduce 
		\begin{align}\label{s4:31}
			\nonumber&\le|-\int_\mathcal{G}u(\Delta_{\mathcal{E}}\Psi)d\mu_{\mathcal{E}}\ri|\\
			\nonumber&\leq\sum_{e\in\mathcal{E}}\int_{0}^{l_e}|u_e(x)||\Psi_e^{\prime\prime}(x)|dx\\
			\nonumber&\leq \frac{C}{R}\sum_{e\in\mathcal{E}}\int_{0}^{l_e}|u_e(x)| e^{\frac{-\delta \tilde{d}(x,x_0)}{R}}\textbf{1}_{\tilde{B}_{R+j}^c}(x)dx\\
			\nonumber&\leq \epsilon \sum_{e\in\mathcal{E}}\int_{0}^{l_e}  V_e(x) |u_e(x)|^{\sigma} e^{\frac{-\delta \tilde{d}(x,x_0)}{R}}\textbf{1}_{\tilde{B}_{R+j}^c}(x) dx+ \frac{C_{\epsilon}}{R^{\frac{\sigma}{\sigma-1}}}\sum_{e\in\mathcal{E}}\int_{0}^{l_e} V_e^{-\frac{1}{\sigma-1}}(x) e^{\frac{-\delta \tilde{d}(x,x_0)}{R}} \textbf{1}_{\tilde{B}_{R+j}^c}(x)dx\\
				\nonumber&\leq \epsilon \sum_{e\in\mathcal{E}}\int_{0}^{l_e}  V_e(x) |u_e(x)|^{\sigma} e^{\frac{-\delta \tilde{d}(x,x_0)}{R}}\textbf{1}_{\tilde{B}_{R+j}^c}(x) dx+ \frac{C_{\epsilon}}{R^{\frac{\sigma}{\sigma-1}}}\sum_{e\in\mathcal{E}}\int_{0}^{l_e} V_e^{-\frac{1}{\sigma-1}}(x) e^{\frac{-\delta d(x,x_0)+\delta j}{R}} \textbf{1}_{\tilde{B}_{R+j}^c}(x)dx\\
			\nonumber&\leq \epsilon C\sum_{e\in\mathcal{E}}\int_{0}^{l_e}  V_e(x) |u_e(x)|^{\sigma} \Psi_e(x)\textbf{1}_{\tilde{B}^c_{R+j}}(x)dx+ \frac{C_{\epsilon,\delta}}{R^{\frac{\sigma}{\sigma-1}}}\sum_{e\in \mathcal{E}}\int_{0}^{l_e} V_e^{-\frac{1}{\sigma-1}}(x) e^{\frac{-\delta d(x,x_0)}{R}} \textbf{1}_{ B_{R}^c}(x)dx\\
			\nonumber&\leq  \epsilon C\sum_{e\in\mathcal{E}}\int_{0}^{l_e}  V_e(x) |u_e(x)|^{\sigma} \Psi_e(x)dx+ \frac{C_{\epsilon,\delta}}{R^{\frac{\sigma}{\sigma-1}}}\sum_{e\in\mathcal{E}\cap B_{R}^c}\int_{0}^{l_e} V_e^{-\frac{1}{\sigma-1}}(x) e^{\frac{-\delta \tilde{d}(x,x_0)}{R}}dx\\
			&\leq \epsilon C\sum_{e\in\mathcal{E}}\int_{0}^{l_e}  V_e(x) |u_e(x)|^{\sigma} \Psi_e(x)dx+ C_{\epsilon,\delta}.
		\end{align}
		Choosing $\epsilon>0$ sufficiently small, combining \eqref{s4:30},  \eqref{s4:31} with \eqref{s4:29}, we obtain 
		\begin{align*}
			\int_{\mathcal{G}}V(x)\Psi(x) |u(x)|^{\sigma}d\mu_{\mathcal{G}}\leq\frac{1}{2} \int_{\mathcal{G}}V(x)\Psi(x) |u(x)|^{\sigma}d\mu_{\mathcal{G}}+ C.
		\end{align*}
		Noting that $\Psi(x)\equiv1$ whenever $x\in \tilde{B}_{R+j}$, we get
		\begin{align*}
			\int_{\tilde{B}_{R+j}}V(x)|u(x)|^{\sigma}d\mu_{\mathcal{G}}\leq\int_{\mathcal{G}}V(x)\Psi(x) |u(x)|^{\sigma}d\mu_{\mathcal{G}}\leq C.
		\end{align*}
		Here, $C$ is a constant  independent of $R$. Hence, the conclusion immediately follows by taking the limit as $R\ra\infty$.
	\end{proof}
	
		Finally, we present the proofs of Theorem \ref{s3:T2}. \\
	
		\noindent$\textbf{\emph{Proof of Theorem \ref{s3:T2}.}}$  Under the conditions of Lemma \ref{s4:L4}, we now show that $u\equiv0$. Making use of H\"{o}lder's inequality and condition \eqref{s3:4}, we estimate the vertex term in \eqref{s4:30} as
	\begin{align}\label{s4:28}
		\nonumber\le|-\int_{\mathcal{G}}u\le(\Delta_\mathcal{V} \Psi \ri)d\mu_\mathcal{V}\ri|&\leq\frac{C}{R}\sum_{x\in\mathcal{V}}\mu_{\mathcal{V}}(x) |u(x)|e^{-\frac{\delta d(x,x_0)}{R}}\textbf{1}_{B_{R}^c}(x)\\
		\nonumber&\leq \frac{C}{R}\le(\sum_{x\in\mathcal{V}}\mu_{\mathcal{V}}(x) |u(x)|^\sigma V(x)e^{-\frac{\delta d(x,x_0)}{R}}\textbf{1}_{B_{R}^c}(x)\ri)^{\frac{1}{\sigma}}\le(\sum_{x\in\mathcal{V}}\mu_{\mathcal{V}}(x) V^{-\frac{1}{\sigma-1}}(x)e^{-\frac{\delta d(x,x_0)}{R}}\textbf{1}_{ B_{R}^c}(x)\ri)^{\frac{\sigma-1}{\sigma}}\\
		\nonumber&\leq C\le(\sum_{x\in\mathcal{V}}\mu_{\mathcal{V}}(x) |u(x)|^\sigma V(x)e^{-\frac{\delta d(x,x_0)}{R}}\textbf{1}_{ B_{R}^c}(x)\ri)^{\frac{1}{\sigma}}\\
		&\leq C\le(\sum_{x\in\mathcal{V}\cap B_{R}^c}\mu_{\mathcal{V}}(x) |u(x)|^\sigma V(x)\ri)^{\frac{1}{\sigma}}.
	\end{align}
	For the edge term, a similar argument using \eqref{s4:31} yields
	\begin{align}\label{s4:32}
		\nonumber&\le|-\int_{\mathcal{G}}u\le(\Delta_\mathcal{E} \Psi \ri)d\mu_\mathcal{E}\ri|\\
		\nonumber&\leq\frac{C}{R}\sum_{e\in\mathcal{E}}\int_{0}^{l_e}|u_e(x)| e^{\frac{-\delta \tilde{d}(x,x_0)}{R}}\textbf{1}_{\tilde{B}_{R+j}^c}(x)dx\\
		\nonumber&\leq \frac{C}{R} \sum_{e\in\mathcal{E}}\le\{\le(\int_{0}^{l_e}V_e(x) |u_e(x)|^\sigma e^{\frac{-\delta \tilde{d}(x,x_0)}{R}}\textbf{1}_{\tilde{B}_{R+j}^c}(x) dx\ri)^\frac{1}{\sigma}\le(\int_{0}^{l_e}V_e^{-\frac{1}{\sigma-1}}(x) e^{\frac{-\delta \tilde{d}(x,x_0)}{R}} \textbf{1}_{ \tilde{B}_{R+j}^c}(x)dx\ri)^\frac{\sigma-1}{\sigma}\ri\}\\
		\nonumber&\leq \frac{C}{R} \le(\sum_{e\in\mathcal{E}}\int_{0}^{l_e}V_e(x) |u_e(x)|^\sigma e^{\frac{-\delta \tilde{d}(x,x_0)}{R}}\textbf{1}_{\tilde{B}_{R+j}^c}(x) dx\ri)^\frac{1}{\sigma}\le(\sum_{e\in\mathcal{E}}\int_{0}^{l_e}V_e^{-\frac{1}{\sigma-1}}(x) e^{\frac{-\delta \tilde{d}(x,x_0)}{R}} \textbf{1}_{ \tilde{B}_{R+j}^c}(x)dx\ri)^\frac{\sigma-1}{\sigma}\\
		\nonumber&\leq  \frac{C}{R} \le(\sum_{e\in\mathcal{E}}\int_{0}^{l_e}V_e(x) |u_e(x)|^\sigma \textbf{1}_{ \tilde{B}_{R+j}^c}(x) dx\ri)^\frac{1}{\sigma}\le(\sum_{e\in \mathcal{E}}\int_{0}^{l_e}V_e^{-\frac{1}{\sigma-1}}(x) e^{\frac{-\delta d(x,x_0)}{R}}  \textbf{1}_{ \tilde{B}_{R+j}^c}(x) dx\ri)^\frac{\sigma-1}{\sigma}\\
		\nonumber&\leq\frac{C}{R} \le(\sum_{e\in\mathcal{E}}\int_{0}^{l_e}V_e(x) |u_e(x)|^\sigma \textbf{1}_{B_R^c}(x) dx\ri)^\frac{1}{\sigma}\le(\sum_{e\in \mathcal{E}\cap B_{R}^c}\int_{0}^{l_e}V_e^{-\frac{1}{\sigma-1}}(x) e^{\frac{-\delta d(x,x_0)}{R}} dx\ri)^\frac{\sigma-1}{\sigma}\\
		& \leq C\le(\sum_{e\in \mathcal{E}\cap B_{R}^c}\int_{0}^{l_e}V_e(x) |u_e(x)|^\sigma dx\ri)^\frac{1}{\sigma}.
	\end{align}
	Substituting the estimates \eqref{s4:28} and \eqref{s4:32} into \eqref{s4:29}, we obtain
	\begin{equation*}
		\int_{\tilde{B}_{R+j}}V(x) |u(x)|^\sigma d\mu_{\mathcal{G}}\leq \int_{\mathcal{G}}V(x)\Psi(x) |u(x)|^{\sigma}d\mu_{\mathcal{G}}\leq C\le(\sum_{x\in\mathcal{V}\cap B_{R}^c}\mu_{\mathcal{V}}(x) |u(x)|^\sigma V(x)\ri)^{\frac{1}{\sigma}}+C\le(\sum_{e\in \mathcal{E}\cap B_{R}^c}\int_{0}^{l_e}V_e(x) |u_e(x)|^\sigma dx\ri)^\frac{1}{\sigma}.
	\end{equation*}
	Letting $R\to\infty$, it then follows from Lemma \ref{s4:L5} that
	$$	\int_{\mathcal{G}}V(x)|u(x)|^\sigma d\mu_\mathcal{G}\leq 0,$$
	which implies that $u\equiv0$ on $\mathcal{G}$. $\hfill\Box$\\

	\noindent	\textbf{Acknowledgements}   The author would like to appreciate the reviewers and editors for their careful reading, constructive comments and helpful suggestions on this paper. This work was supported by the National Natural Science Foundation of China under Grant No.12471088.\\

	\noindent\textbf{Disclosure of interest}
	The authors report there are no competing interests to declare.
	\\
	
	\noindent\textbf{Data availability}
	Data sharing not applicable as no datasets were used or analysed during the current study.
	\\

\end{document}